\documentclass[11pt,psamsfonts]{amsart}

\usepackage{amssymb}
\usepackage{hhline}

\begin{document}

\setlength{\parindent}{0pt} \setlength{\parskip}{1ex}

\newtheorem{Le}{Lemma}[section]
\newtheorem{Th}[Le]{Theorem}

\newtheorem{Thm}{Theorem}

\theoremstyle{definition}
\newtheorem{Def}[Le]{Definition}
\newtheorem{Problem}[Le]{Problem}

\theoremstyle{remark}
\newtheorem{Rem}[Le]{Remark}

\title[Special metrics with $Q^{1,1,1}$ or $M^{1,1,0}$ as principal orbit]{Special cohomogeneity one metrics with $Q^{1,1,1}$
or $M^{1,1,0}$ as principal orbit}
\author{Frank Reidegeld}
\address{Department Mathematik, Bereich Analysis und Differentialgeometrie, Universit\"at Hamburg, Bundesstra\ss e 55, 20146
Hamburg, Germany}
\email{frank.reidegeld\char"40 math.uni-hamburg.de}
\subjclass[2000]{Primary 53C29, Secondary 53C10, 53C25,\\ \parbox{9pt}{$ $} 53C44. \medskip \\  
\parbox{9pt}{$ $} This work was supported by the SFB 676 of the Deutsche
Forschungsgemeinschaft.\\}

\begin{abstract}
We classify all cohomogeneity one manifolds with principal orbit
$Q^{1,1,1}=SU(2)^3/U(1)^2$ or $M^{1,1,0}=(SU(3)\times
SU(2))/(SU(2)\times U(1))$ whose holonomy is contained in
$\text{Spin}(7)$. Various metrics with different kinds of
singular orbits can be constructed by our methods. It turns out
that the holonomy of our metrics is automatically $SU(4)$ and 
that they are asymptotically conical. Moreover, we investigate 
the smoothness of the metrics at the singular orbit.
\end{abstract}

\maketitle

\section{Introduction}

The subject of this article are eight-dimensional manifolds with
special holonomy. These manifolds are not only of interest for
purely mathematical reasons, but they are also studied in super
string theory (cf. Acharya, Gukov \cite{Ach}, Cveti\v{c} et al.
\cite{Cve01},\cite{Cve},\cite{Cve03}, Gukov, Sparks \cite{Guk}).
Our aim is to classify all metrics of a special type whose
holonomy is contained in $\text{Spin}(7)$. Explicit metrics with
holonomy $\text{Spin}(7)$ or $SU(4)$ are hard to construct. This
problem becomes a lot of easier if we assume that the metric is
preserved by a cohomogeneity one action. In
this situation, our task is equivalent to solving a system of ordinary
differential equations. Examples of cohomogeneity one
metrics with holonomy $\text{Spin}(7)$ or $SU(4)$ can be found in
Bazaikin \cite{Baz}, Bryant, Salamon \cite{BrS}, Cveti\v{c} et al.
\cite{Cve01}, \cite{Cve02},\cite{Cve},\cite{Cve03}, and Herzog, Klebanov
\cite{Herzog}.

If we assume that the group which acts with cohomogeneity one is compact,
the number of possible principal orbits is countable \cite{Rei1}. Furthermore, the
space of all homogeneous $G_2$-structures on a fixed principal orbit is finite-dimensional. This fact makes it possible to
obtain partial classification results. In this article, we assume that the principal orbit is of a certain kind. First, we
consider coset spaces of type $SU(2)^3/U(1)^2$ which are denoted by $Q^{k,l,m}$. The indices $k$, $l$, and $m$ describe the
embedding of $U(1)^2$ into $SU(2)^3$. Second, we investigate quotients of type $(SU(3)\times SU(2))/(SU(2)\times U(1))$ where
the semisimple part of $SU(2)\times U(1)$ is embedded into the first factor of $SU(3)\times SU(2)$. These spaces are denoted by
$M^{k,l,0}$. If the third index is non-zero, we obtain a space which is covered by $M^{k,l,0}$. The meaning of the indices
$k$ and $l$ will be explained in Section \ref{M110}. At the singular orbit, a space of cohomogeneity one may have a singularity.
In this article, we allow orbifold singularities but exclude all other ones. Our results can be summed up as follows:

\begin{Thm}
Let $(M,\Omega)$ be a cohomogeneity one $\text{Spin}(7)$-manifold
whose principal orbits are of type $Q^{k,l,m}$. In this situation, the
following statements are true:

\begin{enumerate}
    \item The principal orbits are $SU(2)^3$-equivariantly diffeomorphic to $Q^{1,1,1}$.
    \item The metric $g$ which is associated to $\Omega$ has holonomy $SU(4)$.
    \item If $M$ has a singular orbit, it has to be $S^2\times S^2$ or $S^2\times S^2\times S^2$.
    Any $SU(2)^3$-invariant metric on the singular orbit can be uniquely extended to a complete
    cohomogeneity one metric with holonomy $SU(4)$. For any choice of the singular orbit and its
    metric, $(M,g)$ is asymptotically conical. If the singular orbit is $S^2\times S^2$, $g$ is
    a smooth metric. In the other case, the metric cannot be smooth at the singular orbit.
\end{enumerate}
\end{Thm}

\begin{Thm}
Let $(M,\Omega)$ be a cohomogeneity one $\text{Spin}(7)$-orbifold whose principal orbit is of type $M^{k,l,0}$. In this
situation, the following statements are true:

\begin{enumerate}
    \item The principal orbit is $SU(3)\times SU(2)$-equivariantly diffeomorphic to $M^{1,1,0}$.
    \item The metric $g$ which is associated to $\Omega$ has holonomy $SU(4)$.
    \item If $M$ has a singular orbit, it has to be $S^2$, $\mathbb{CP}^2$, or $S^2\times\mathbb{CP}^2$.
    Any $SU(3)\times SU(2)$-invariant metric on the singular orbit
    can be uniquely extended to a complete cohomogeneity one metric with
    holonomy $SU(4)$. For any choice of the singular orbit and its
    metric, $(M,g)$ is asymptotically conical. If the singular orbit is $S^2$
    or $S^2\times \mathbb{CP}^2$, $g$ cannot be a smooth orbifold metric near the
    singular orbit. In the second case, $g$ is always smooth.
\end{enumerate}
\end{Thm}

The metrics which we construct are also described in Cveti\v{c} et
al. \cite{Cve01},\cite{Cve03} and Herzog, Klebanov \cite{Herzog}.
These works are based on methods by Berard-Bergery \cite{BerBer},
Page, Pope \cite{PagePope}, and Stenzel \cite{Stenzel}.
Nevertheless, our proofs that there are no further metrics of the
above kind and that the holonomy automatically reduces to $SU(4)$
are new results. In particular, we prove that we
cannot deform the metrics from the literature into metrics with
holonomy $\text{Spin}(7)$ without loosing the $SU(2)^3$- or
$SU(3)\times SU(2)$-symmetry. Moreover, the smoothness of the
metrics at the singular orbit and the global shape of our manifolds
are studied in detail.

This article is organized as follows: In Section \ref{2ndSection}
and \ref{CohomOneSec}, we collect some facts on metrics with
exceptional holonomy and on cohomogeneity one manifolds. The
metrics with principal orbit $Q^{k,l,m}$ are studied in Section
\ref{Q111} and those with principal orbit $M^{k,l,0}$ in Section
\ref{M110}.

\section{Metrics with holonomy Spin($7$)}
\label{2ndSection}

In this section, we will collect some general facts on $\text{Spin}(7)$-manifolds. Later on, we have to deal not
only with the $\text{Spin}(7)$-structure on the manifold but also with the induced $G_2$-structure on the principal orbits.
Therefore, we first introduce the group $G_2$:

\begin{Def}
Let $\mathbb{O}$ be the division algebra of the \emph{octonions}. An $\mathbb{R}$-linear non-zero map $\phi: \mathbb{O}
\rightarrow \mathbb{O}$ satisfying $\phi(x\cdot y) = \phi(x)\cdot \phi(y)$ for all $x,y\in\mathbb{O}$ is called an
\emph{automorphism of $\mathbb{O}$}. The group of all automorphisms of $\mathbb{O}$ we call $G_2$ and its Lie algebra
$\mathfrak{g}_2$.
\end{Def}

\begin{Le}
\begin{enumerate}
    \item Any automorphism of $\mathbb{O}$ fixes $1$ and leaves its orthogonal complement $\text{Im}(\mathbb{O})$ invariant.
    This yields an irreducible seven-dimensional representation of $G_2$, which we call the \emph{standard representation of
    $G_2$}.
    \item The group which is generated by the left multiplications with unit octonions is isomorphic to $\text{Spin}(7)$. Its
    Lie algebra we denote by $\mathfrak{spin}(7)$. The action of this algebra on $\mathbb{O}$ is equivalent to the spinor
    representation of $\mathfrak{so}(7)$.
\end{enumerate}
\end{Le}

We rewrite the canonical basis $(1,i,j,k,\epsilon,i\epsilon,j\epsilon,k\epsilon)$ of $\mathbb{O}$ as $(e_0,\ldots,e_8)$.
Furthermore, we define the one-forms $dx^i$ by $dx^{i}(e_j):=\delta^i_j$ and $dx^{i_1\ldots i_k}$ as
$dx^{i_1}\wedge\ldots\wedge dx^{i_k}$. The three-form

\begin{equation}
\label{omega}
\omega:=dx^{123}+dx^{145}-dx^{167}+dx^{246}+dx^{257}+dx^{347}-dx^{356}
\end{equation}

is invariant under $G_2$ and the four-form

\begin{equation}
\label{Omega}
\begin{split}
\Omega & :=
dx^{0123}+dx^{0145}-dx^{0167}+dx^{0246}+dx^{0257}+dx^{0347}-dx^{0356}\\
&\quad
-dx^{1247}+dx^{1256}+dx^{1346}+dx^{1357}-dx^{2345}+dx^{2367}
+dx^{4567}\:.\\
\end{split}
\end{equation}

is invariant under $\text{Spin}(7)$. Up to constant multiples, the above forms are the only elements of $\bigwedge^3
\text{Im}(\mathbb{O})^{\ast}$ ($\bigwedge^4 \mathbb{O}^{\ast}$) with these properties. $\Omega$ and $\omega$ are related by the
formula

\begin{equation}
\Omega=\ast\omega+dx^0\wedge\omega\:,
\end{equation}

where $\ast$ is the Hodge star with respect to the canonical metric on $\text{Im}(\mathbb{O})$. We are now able to define the
notion of a $G_2$- ($\text{Spin}(7)$-)structure:

\begin{Def}
\begin{enumerate}
    \item Let $M$ be a seven-dimensional manifold and $\omega$ be a three-form on $M$. We assume that for any $p\in M$ there
    exists an open neighborhood $U$ of $p$ and a local frame $(X_i)_{1\leq i\leq 7}$ on $U$ with the following property:
    $\omega$ has with respect to $(X_i)_{1\leq i\leq 7}$ the same coefficients as the three-form (\ref{omega}) with respect to
    $(e_i)_{1\leq i\leq 7}$. In this situation, $\omega$ is called a \emph{$G_2$-structure} on $M$.
    \item Let $M$ be an eight-dimensional manifold and $\Omega$ be a four-form on $M$. We assume that for any $p\in M$ there
    exists an open neighborhood $U$ of $p$ and a local frame $(X_i)_{0\leq i\leq 7}$ on $U$ with the following property:
    $\Omega$ has with respect to $(X_i)_{0\leq i\leq 7}$ the same coefficients as the four-form (\ref{Omega}) with respect to
    $(e_i)_{0\leq i\leq 7}$. In this situation, $\Omega$ is called a \emph{$\text{Spin}(7)$-structure} on $M$.
\end{enumerate}
\end{Def}

Alternatively, we could have defined a $G_2$-
($\text{Spin}(7)$-)structure as a certain principal bundle with
structure group $G_2$ ($\text{Spin}(7)$). On any manifold with a
$G_2$- ($\text{Spin}(7)$-)structure there exist an
\emph{associated metric and orientation}, which depend on $\omega$
($\Omega$) only. The $G_2$- ($\text{Spin}(7)$)-structures can be
divided according to their intrinsic torsion into $16$ classes
(cf. Fern\'andez, Gray \cite{Fer} for further details). For our
considerations, we only need the following of those classes:

\begin{Def}
\begin{enumerate}
    \item A $G_2$-structure $\omega$ is called \emph{nearly parallel} if there exists a
    $\lambda\in\mathbb{R}\setminus\{0\}$ such that $d\omega=\lambda\ast\omega$. "$\ast$" denotes the Hodge star operator with
    respect to the associated metric and orientation.
    \item A $G_2$-structure $\omega$ is called \emph{cosymplectic} if $d\ast\omega=0$.
    \item A $\text{Spin}(7)$-structure $\Omega$ is called \emph{parallel} if $\nabla^g\Omega=0$, where $g$
    is the associated metric and $\nabla^g$ is the Levi-Civita connection of $g$. In this situation, we call $(M,\Omega)$ a
    \emph{$\text{Spin}(7)$-manifold.}
\end{enumerate}
\end{Def}

If $\Omega$ is a $\text{Spin}(7)$-structure, $\nabla^g\Omega=0$ is equivalent to $d\Omega=0$ (cf. Fern\'andez \cite{Fer2}).
Since it is more convenient to work with the latter equation in practical situations, we will often use it instead of
$\nabla^g\Omega=0$. $\text{Spin}(7)$-manifolds have the following interesting properties:

\begin{Th} (Cf. Bonan \cite{Bon} and Wang \cite{Wang}.)
\begin{enumerate}
    \item The associated metric is Ricci-flat.
    \item The holonomy of the metric is contained in $\text{Spin}(7)$.
\end{enumerate}
\end{Th}

\begin{Rem}
Since $SU(4)$ can be embedded into $\text{Spin}(7)$, manifolds whose holonomy is contained in $SU(4)$ are a special case of
$\text{Spin}(7)$-manifolds. The holonomy of a metric on an eight-dimensional manifold is a subgroup of $SU(4)$ if  and only if
it is Ricci-flat and K\"ahler.
\end{Rem}

\section{Geometrical structures of cohomogeneity one}
\label{CohomOneSec}

In this section, we motivate why it is promising to assume that the metric is invariant under a cohomogeneity one action and
give a short introduction into the issue of Spin($7$)-manifolds of cohomogeneity one.

Since there are certain non-linear restrictions on $\Omega$, $d\Omega=0$ should be viewed as a non-linear partial
differential equation. This makes it extraordinary difficult to construct examples of $\text{Spin}(7)$-manifolds. The first
local examples have been constructed by Bryant \cite{Br}, the first complete ones by Bryant and Salamon \cite{BrS}, and the
existence of compact manifolds with holonomy $\text{Spin}(7)$ has been proven by Joyce \cite{Joy}. If $\Omega$ is preserved by
a large group acting on the manifold, the equation $d\Omega=0$ simplifies. The optimal case would be where the group action is
transitive. Unfortunately, there are no interesting homogeneous examples:

\begin{Th}
(See Alekseevskii and Kimelfeld \cite{Alki}.) Any homogeneous Ricci-flat metric is necessary flat.
\end{Th}

The next case which is natural to consider is where the group acts by cohomogeneity one:

\begin{Def}
\begin{enumerate}
    \item Let $(M,g)$ be a connected Riemannian manifold with an isometric action by a Lie group $G$. An orbit $\mathcal{O}$
    of this action is called a \textit{principal orbit} if there is an open subset $U$ of $M$ with the following properties:
    $\mathcal{O}\subseteq U$ and $U$ is $G$-equivariantly diffeomorphic to $\mathcal{O}\times V$, where
    $V\subseteq\mathbb{R}^n$ is an open set.
    \item In the above situation, $\dim{V}$ (or equivalently $\dim{M}-\dim{\mathcal{O}}$) is called the \emph{cohomogeneity of
    the $G$-action on $M$}.
    \item A $\text{Spin}(7)$-manifold $(M,\Omega)$ is called \emph{of cohomogeneity one} if there exists a cohomogeneity one
    action which preserves $\Omega$ (and thus the associated metric).
\end{enumerate}
\end{Def}

Since any two principal orbits are $G$-equivariantly diffeomorphic, the cohomogeneity is a well-defined number. Any principal
orbit can be identified with a coset space $G/H$. Moreover, any non-principal orbit can be identified with a coset space $G/K$
and after conjugation we can assume that $H\subseteq K \subseteq G$. The group $K$ cannot be chosen arbitrarily:

\begin{Th} \label{Most} (See Mostert \cite{Mostert}.) Let $(M,g)$ be a Riemannian manifold with an isometric cohomogeneity one
action by a Lie group $G$. Furthermore, let all principal orbits be $G$-equivariantly diffeomorphic to $G/H$ with $H \subseteq
G$. We assume that there exists a non-principal orbit which we identify with $G/K$ where $H\subseteq K\subseteq G$. In this
situation, $K/H$ is diffeomorphic to a sphere.
\end{Th}

\begin{Rem}
\begin{enumerate}
    \item If $K/H$ is a quotient of a sphere by a finite group, $M$ is not a manifold, but still an orbifold. Since we are
    also interested in $\text{Spin}(7)$-manifolds with "nice" singularities, we will include such spaces into our
    considerations, too. On the following pages, we will assume that $M$ is a manifold. Nevertheless, we can easily modify our
    statements such that they are valid in the orbifold-case, too.
    \item Since the volume of $K/H$ shrinks to zero as we approach the singular orbit, we will refer to $K/H$ as the
    \emph{collapsing sphere}.
\end{enumerate}
\end{Rem}

For the rest of this section, let $(M,\Omega)$ be a $\text{Spin}(7)$-manifold of cohomogeneity one. $G$, $H$, and $K$ shall
denote the same groups as above. We will assume from now on that $G$ is compact. We denote the Lie algebras of $G$, $H$, and
$K$ by $\mathfrak{g}$, $\mathfrak{h}$, and $\mathfrak{k}$. Beside Theorem \ref{Most} there are further restrictions on the
shape of $M$:

It is a well-known fact (cf. Mostert \cite{Mostert}) that $M/G$ is
homeomorphic to $S^1$, $\mathbb{R}$, $[0,1]$, or $[0,\infty)$. In
the first case, $M$ would have an infinite fundamental group and
the holonomy thus would be a subgroup of $G_2$. We will therefore
exclude this case. In the second case, $M$ would contain a line.
It would follow from the Cheeger-Gromoll splitting theorem that
$M$ is a Riemannian product of $\mathbb{R}$ and another space.
This case we will not consider, too. In the third case, $M$ would
be compact. Since its Ricci curvature would be non-positive, all
Killing vector fields would be parallel and commute with each
other. The principal orbit thus would be a flat torus. Since in
that situation, $(M,g)$ would be flat, too, we can assume that
$M/G$ is $[0,\infty)$. The point $0$ corresponds to a
non-principal orbit and all other orbits are principal ones.

If $K/H=S^0=\mathbb{Z}_2$, the non-principal orbit is called \emph{exceptional}. In that case, $M$ would be twofold covered by
a space $\widetilde{M}$ with $\widetilde{M}/G=\mathbb{R}$. We will therefore restrict ourselves to the case where there is
exactly one non-principal orbit with $\dim{K} - \dim{H} > 1$. An orbit which satisfies this requirement is called a
\emph{singular orbit}.

The pull-back of the inclusion $\imath$ of a principal orbit into $M$ yields a four-form $\imath^\ast(\Omega)$ on any
principal orbit. Its Hodge-dual with respect to the restricted metric and a fixed orientation is a three-form which can be
proven to be a $G_2$-structure $\omega$. From the equation $d\Omega=0$ it follows that $d \ast \omega = 0$. Conversely, any
$G$-invariant cosymplectic $G_2$-structure on $G/H$ can be extended to a parallel $\text{Spin}(7)$-structure on $G/H \times
(-\epsilon,\epsilon)$:

\begin{Th}
Let $G/H$ be a seven-dimensional homogeneous space which carries a
$G$-invariant cosymplectic $G_2$-structure $\widetilde{\omega}$.
Then there exists an $\epsilon>0$ and a one-parameter family
$(\omega_t)_{t\in(-\epsilon,\epsilon)}$ of $G$-invariant
three-forms on $G/H$ such that the initial value problem

\begin{eqnarray}
\label{Evol} \frac{\partial}{\partial t}\ast\omega_t & = &
d_{\scriptscriptstyle G/H}\omega_t\\ \label{InitialEvol} \omega_0
& = & \widetilde{\omega}
\end{eqnarray}

has a unique solution on $G/H\times(-\epsilon,\epsilon)$ with the following properties:

\begin{enumerate}
    \item $\omega_t$ is a $G_2$-structure on $G/H\times\{t\}$.
    \item $d_{\scriptscriptstyle G/H}\ast\omega_t=0$.
    \item $\ast\omega_t$ is in the same cohomology class in $H^4(M,\mathbb{R})$
    as $\ast\widetilde{\omega}$.
\end{enumerate}

In the above formula, $\tfrac{\partial}{\partial t}$ denotes the Lie derivative in $t$-direction. The index $G/H$ of $d$
emphasizes that we consider the exterior derivative on $G/H\times\{t\}$ instead of $G/H\times (-\epsilon,\epsilon)$. In the
above situation, the four-form $\Omega:=\ast\omega+dt \wedge\omega$ is a $G$-invariant parallel $\text{Spin}(7)$-structure on
$G/H \times(-\epsilon,\epsilon)$.

Conversely, let $\Omega$ be a parallel $\text{Spin}(7)$-structure preserved by a cohomogeneity one action of a compact Lie
group $G$. The isotropy group of the $G$-action on the principal orbit we denote by $H$. We identify the union of all
principal orbits $G$-equivariantly with $G/H\times I$, where the interval $I$ is parameterized by arclength. In this
situation, the $G_2$-structures on the principal orbits are cosymplectic and satisfy equation (\ref{Evol}).
\end{Th}

\begin{Rem}
\begin{enumerate}
    \item The above theorem has been proven by Hitchin \cite{Hitch} in a more general context: In \cite{Hitch}, $G/H$ is
    replaced by a seven-dimensional compact manifold which carries a cosymplectic $G_2$-struc\-ture but is not necessarily
    homogeneous.
    \item If $\omega$ is nearly parallel, the maximal solution of (\ref{Evol}) describes a cone over $G/H$. More precisely,
    the space of cohomogeneity one is isometric to a cone over $G/H$ which carries the metric associated to $\omega$.
    \item Since $(M,\Omega)$ is of cohomogeneity one, the equation $\tfrac{\partial}{\partial t}\ast\omega_t = d\omega_t$ is
    equivalent to a system of ordinary differential equations and thus easier to handle than the equation $d\Omega=0$ in the
    general situation. Since we assume that $M$ has at least one singular orbit, we try to fix the initial conditions on a
    singular orbit $G/K$. Since $\dim{G/K}<7$, the differential equations will degenerate at the singular orbit.
\end{enumerate}
\end{Rem}

Before we can make the initial value problem (\ref{Evol}),
(\ref{InitialEvol}) explicit, we have to fix a principal and a
singular orbit. The possible principal orbits are exactly those
homogeneous spaces $G/H$ which admit at least one $G$-invariant
cosymplectic $G_2$-structure $\omega$. From the $G$-invariance of
$\omega$ it follows that the isotropy action of $H$ on the
tangent space has to be equivalent to the action of a subgroup of
$G_2$ on $\text{Im}(\mathbb{O})$. Conversely, there always exists
a $G$-invariant $G_2$-structure on $G/H$ if $H$ has the
above property: $G$ can be considered as an $H$-bundle over
$G/H$. Its extension to a principal bundle with structure group
$G_2$ is the $G_2$-structure we search for. We sum up our
observations to the following lemma:

\begin{Le} \label{G2HomLemma}
Let $G/H$ be a homogeneous space such that $G$ acts effectively on $G/H$ and let $p\in G/H$ be arbitrary. We identify $H$ with
its isotropy representation on $T_p G/H$ and $G_2$ with its standard representation. $G/H$ admits a $G$-invariant
$G_2$-structure if and only if there exists a vector space isomorphism $\varphi: T_p G/H \rightarrow \text{Im}(\mathbb{O})$
such that $\varphi H \varphi^{-1} \subseteq G_2$.
\end{Le}

For our calculations, we not only need the existence of one particular $G$-invariant cosymplectic $G_2$-structure. More
precisely, we need a set of $G_2$-structures such that the flow equation $\tfrac{\partial}{\partial t}\ast\omega_t = d\omega_t$
does not leave this set. The set of all $G$-invariant cosymplectic $G_2$-structures clearly satisfies this condition. On the
principal orbits which we consider in this article we are able to describe that set explicitly. We will see
that it is not difficult to classify all $G$-invariant metrics on $G/H$. A Riemannian metric on a seven-dimensional manifold
together with an orientation  is the same as an $SO(7)$-structure. We thus can replace the problem of finding all
$G$-invariant $G_2$-structures by the following simpler one:

\begin{Problem} \label{G2StrucProblem}
Let $\mathcal{G}$ be an arbitrary $G$-invariant $SO(7)$-structure on $G/H$. Find all $G$-invariant $G_2$-structures on $G/H$
whose extension to a principal bundle with structure group $SO(7)$ is $\mathcal{G}$.
\end{Problem}

This problem can be solved by purely algebraic methods: Since $G$ acts transitively on $G/H$, a $G$-invariant $G_2$-structure
$\omega$ is determined by a basis of a tangent space $T_p G/H$ which can be identified with the basis
$(i,j,k,\epsilon,i\epsilon,j\epsilon,k\epsilon)$ of $\text{Im}(\mathbb{O})$. Any other $G$-invariant $G_2$-structure can be
identified with another basis or equivalently with a linear map $\phi:T_p G/H \rightarrow T_p G/H$ which maps the first basis
into the second one. The condition that $\omega$ is $G$-invariant translates into $H\phi = \phi H$, where $H$ is as before
identified with its isotropy representation on $T_p G/H$. If we identify $T_p G/H$ with $\text{Im}(\mathbb{O})$, this
condition translates into $\phi\in \text{Norm}_{GL(7)} H:=\{\phi\in GL(7)|\phi H\phi^{-1} = H\}$. The group of all $\phi$
which leave the extension of the $G_2$-structure to an $SO(7)$-structure invariant is $\text{Norm}_{SO(7)} H$. Since the
$G_2$-structure is stabilized by $\text{Norm}_{G_2} H$, the set we search for in Problem \ref{G2StrucProblem}, can be
described as $\text{Norm}_{SO(7)} H/\text{Norm}_{G_2} H$:

\begin{Le} \label{G2StrucSpace} Let $G/H$ be a seven-dimensional homogeneous space. We assume that $G$ acts effectively and
that $G/H$ admits a $G$-invariant $G_2$-structure. The space of all $G$-invariant $G_2$-structures on $G/H$ which have a fixed
associated metric and orientation is $\text{Norm}_{SO(7)} H$-equivariantly diffeomorphic to:

\begin{equation}
\text{Norm}_{SO(7)}H/\text{Norm}_{G_2}H\:.
\end{equation}

In particular, this space does not depend on the choice of the $G$-invariant metric and the orientation.
\end{Le}

By solving $d\ast\omega=0$ we are able to parameterize the space
of all $G$-invariant cosymplectic $G_2$-structures on $G/H$ and to
transform the equation $\tfrac{\partial}{\partial t}\ast\omega_t =
d\omega_t$ into an explicit system of ordinary differential
equations. We assume that we have obtained a solution of those
equations which has a singular orbit. In that situation, we are
not done yet, since we have to prove that the metric $g$ and the
four-form $\Omega$ can be smoothly extended from the union of all
principal orbits to the singular orbit. A set of necessary and
sufficient smoothness conditions for arbitrary tensor fields of
cohomogeneity one can be found in Eschenburg, Wang \cite{Esch}.
Before we can state the theorem of Eschenburg and Wang, we have to
introduce some notation.

The principal orbit $G/H$ is a sphere bundle over the singular
orbit $G/K$. There is an appropriate parameterize of $G/K$, a so
called \emph{tubular parameterize}, which is a disc bundle over
$G/K$. The tangent space of any point of $G/K$ can therefore be
splitted into a horizontal and a vertical part. As a $K$-module,
the horizontal part is the same as the complement $\mathfrak{p}$
of $\mathfrak{k}\subseteq\mathfrak{g}$ with respect to an
$\text{Ad}_K$-invariant metric. The orbits of the $K$-action on
the vertical part $\mathfrak{p}^\perp$ are, except $\{0\}$,
spheres of type $K/H$. Let $\rho$ be a $G$-invariant tensor field
with values in a vector bundle $\mathcal{B}$ which is defined
everywhere on the tubular parameterize except on $G/K$. Its
extension to the singular orbit is determined by its values at an
arbitrary point of $G/K$ if it exists.

On any cohomogeneity one manifold there exists a geodesic $\gamma$
which intersects all orbits perpendicularly. We assume that
$\gamma(0)\in G/K$ and that $\gamma$ is parameterized by
arclength. The action of $K$ on such a geodesic generates a fiber
of the disc bundle. Therefore, it suffices to consider $\rho$
along $\gamma$ only. Any metric whose holonomy is contained in
$\text{Spin}(7)$ is Einstein. Since any Einstein metric is
analytic \cite{DeKa}, we assume that $\rho$ is a power series with
respect to the parameter of $\gamma$. The $m^{th}$ derivative of
$\rho$ in the vertical direction can be considered as a map, which
assigns to a tuple $(v_1,\ldots,v_m)\in \mathfrak{p}^\perp$ an
element

\begin{equation}
\left.\frac{\partial^m}{\partial v_1\ldots\partial v_m}\right|_{\gamma(0)} \rho
\end{equation}

of $\mathcal{B}_p$. This map can be
extended to a map $\psi_m:S^m(\mathfrak{p}^\perp) \rightarrow
\mathcal{B}_p$, where $S^m$ denotes the symmetric power. Since
$\rho$ is analytic in the above sense, it is determined by the
$\psi_m$. If $\rho$ has a smooth extension to the singular orbit,
the $\psi_m$ have to be $K$-equivariant. This necessary
condition is in fact sufficient, too:

\begin{Th} \label{SmoothExtension} (Cf. Eschenburg, Wang \cite{Esch}.)
Let $(M,g)$ be Riemannian manifold with an isometric action of cohomogeneity one by a Lie group $G$. We assume that this
action has a singular orbit. The isotropy group of the $G$-action at the singular orbit will be denoted by $K$. Let
$\mathcal{B}\subseteq \bigotimes^{s_1} TM\otimes \bigotimes^{s_2} T^\ast M$ be a vector bundle over $M$ whose fibers at the
singular orbit are $K$-equivariantly isomorphic to a $K$-module $B$. Let $r:(0,\varepsilon)\rightarrow B$, where
$\varepsilon>0$, be a real analytic map with Taylor expansion $\sum_{m=1}^\infty r_m t^m$. By the construction which we have
described above, we can identify $r$ with a tensor field $\rho$. This tensor field is defined on a tubular neighborhood of the
singular orbit, but not on the singular orbit itself. $\rho$ is well-defined and has a smooth extension to the singular orbit
if and only if

\begin{equation}
r_m\in \imath_m(W_m)\quad\forall m\in\mathbb{N}_0\:.
\end{equation}

In the above formula, $W_m$ is the space of all $K$-equivariant maps:

\begin{equation}
W_m:=\{P:S^m(\mathfrak{p}^\perp)\rightarrow B |P\:\text{is linear and $K$-equivariant}\}
\end{equation}

and $\imath_m$ is the evaluation map

\begin{equation}
\begin{split}
\imath_m & :W_m\rightarrow B\\ \imath_m(P) & :=
P(\gamma'(0))\:.\\
\end{split}
\end{equation}
\end{Th}

The metrics $g$ which we consider in this article have no "mixed coefficients", i.e.

\begin{equation}
g\in S^2(\mathfrak{p}) \oplus S^2(\mathfrak{p}^\perp)\:.
\end{equation}

In order to check the smoothness of $g$, it suffices to consider the subspaces

\begin{equation}
\begin{split}
W^h_m & :=\{P:S^m(\mathfrak{p}^\perp)\rightarrow S^2(\mathfrak{p}) |P\:\text{is linear and $K$-equivariant}\}\quad\text{and} \\
W^v_m & :=\{P:S^m(\mathfrak{p}^\perp)\rightarrow S^2(\mathfrak{p}^\perp) |P\:\text{is linear and $K$-equivariant}\}\\
\end{split}
\end{equation}

of $W_m$. We describe certain elements of $W^h_m$ and $W^v_m$.
This description yields a sufficient smoothness condition for the
metric which can be easily checked. Let $q^h\in S^2(\mathfrak{p})$
be $K$-invariant. We denote the derivation in the direction of
$\gamma'(0)$ by $\tfrac{\partial}{\partial t}$. Let $n:= \dim{K} -
\dim{H}$ and let $(e_1,\ldots,e_n)$ be an orthonormal basis of
$\mathfrak{p}^\perp$ with respect to a $K$-invariant metric $q^v$
on $\mathfrak{p}^\perp$. Since $K/H$ has to be a distance sphere,
$q^v$ is unique up to a constant factor, which is fixed by
$\|\tfrac{\partial}{\partial t}\|=1$. $\tfrac{\partial}{\partial
t}$ extends by the action of $K$ to a vector field in the radial
direction and we have

\begin{equation}
\tfrac{\partial}{\partial t} \otimes \tfrac{\partial}{\partial t} = e_1\otimes e_1 + \ldots + e_n\otimes e_n\:.
\end{equation}

We define

\begin{equation}
\begin{split}
& \phi^h_{2m} : S^{2m}(\mathfrak{p}^\perp) \rightarrow S^2(\mathfrak{p}) \\
& \phi^h_{2m} ((e_1\otimes e_1 + \ldots + e_n\otimes e_n)^{\otimes m}) := q^h\:. \\
\end{split}
\end{equation}

On the orthogonal complement of $(e_1\otimes e_1 + \ldots + e_n\otimes e_n)^{\otimes m}$, $\phi^h_{2m}$ shall vanish.
$\phi^h_{2m}$ is obviously $K$-equivariant. Theorem \ref{SmoothExtension} tells us that the restriction of the metric to
$S^2(\mathfrak{p})$ is smooth if all odd derivatives vanish and the even derivatives are described by a map of type
$\phi^h_{2m}$. We are now going to interpret this fact geometrically. On the union of all principal orbits we have

\begin{equation}
g = g_t + dt^2\:,
\end{equation}

where $g_t$ is a $t$-dependent $G$-invariant metric on $G/H$. The
tangent space of $G/H$ can be identified with the complement
$\mathfrak{m}$ of $\mathfrak{h}$ in $\mathfrak{g}$. This
identification allows us to consider $\mathfrak{p}$ as a subset of
$T_p G/H$. Our statement on $\phi^h_{2m}$ simply means that the
horizontal part of the metric is smooth if it is invariant under
the action of $K$ on $S^2(\mathfrak{p})$ and if $t\mapsto
g_t(v,v)$ is an even analytic function for all $v\in\mathfrak{p}$.

For the vertical part of the metric, we define analogous maps
$\phi^v_{2m}:S^{2m}(\mathfrak{p}^\perp)$ $\rightarrow
S^2(\mathfrak{p}^\perp)$. If $m=0$, $\phi^v_{2m}$ has to assign
$q^v$ to $1\in\mathbb{R}$. We translate the uniqueness of $q^v$
into a geometrical condition on $g$: Since the vertical part of
$g$ has to coincide with $q^v$ up to $0^{th}$ order, the metric on
the collapsing sphere $K/H$ has to approach the metric of a round
sphere with radius $t$. The length of any great circle on the
collapsing sphere has to be $2\pi t + O(t^2)$ for small $t$. We
denote this length by $\ell (t)$ and have finally found the
smoothness condition $\ell'(0)=2\pi$.

Next, we construct the $\phi^v_{2m}$ for $m\geq 1$. Without loss
of generality, let $e_1=\gamma'(0)$. $\text{span}(e_2,\ldots,e_n)$
is an $H$-module and can be identified with the tangent space of
the collapsing sphere at $\gamma(t)$. There is a one-to-one
correspondence between the $H$-invariant symmetric bilinear forms
on $\text{span}(e_2,\ldots,e_n)$ and the $K$-invariant sections of
$S^2(T^\ast K/H)$. The second derivative of the metric in the
direction $e_1$ has to be such a bilinear form, which we denote by
$q_1^v$. By the action of $K$ we can transform $e_1$ into any
other direction in $\mathfrak{p}^\perp$. We therefore define the
$K$-equivariant map $\phi^v_2$ by

\begin{equation}
\phi^v_2\left( (L_{k\ast} e_1) \otimes (L_{k\ast} e_1)\right) :=
L^\ast_k q^v_1\quad\forall k\in K\:.
\end{equation}

In the above formula, $L_{k\ast}$ denotes the push-forward of the
left-multiplication by $k$ and $L^\ast_k$ the pull-back.
Analogously, to the horizontal case, we define

\begin{equation}
\phi^v_{2m}\left( (e_1\otimes e_1 + \ldots + e_n\otimes
e_n)^{\otimes m-1} \vee (L_{k\ast} e_1) \otimes (L_{k\ast}
e_1)\right) := L^\ast_k q^v_m\quad\forall k\in K\:,
\end{equation}

where $\vee$ denotes the symmetrized tensor product and the
$q_m^v$ are arbitrary $H$-invariant symmetric bilinear forms on
$\text{span}(e_2,\ldots,e_n)$. The fact that the maps
$\phi^v_{2m}$ are for any choice of $q^v_m$ $K$-equivariant makes
a statement on the even derivatives of the vertical part of the
metric. If we write our metric as $g_t + dt^2$, the restriction of
$g_t$ to the vertical directions describes the metric on a
shrinking sphere. In other words, we write the metric in "polar
coordinates" rather than in "Euclidean" ones. A $2m^{th}$
derivative of the metric on $\mathfrak{p}^\perp$ at the origin
therefore corresponds to a $(2m+1)^{st}$ derivative of the
vertical part of $g_t$. Theorem \ref{SmoothExtension} states that
the vertical part of the metric is smooth if all even derivatives
of $g_t$ vanish and all odd ones are described by maps of type
$\phi^v_{2m}$. More explicitly, the vertical part is smooth if

\begin{enumerate}
    \item the restriction of $g_t$ to $\text{span}(e_2,\ldots,e_n)$ is $H$-invariant for all $t$,
    \item the values of $\tfrac{\partial}{\partial t}|_{t=0}g_t(e_i,e_j)$ for $i,j\in\{2,\ldots,n\}$ make $\ell'(0)=2\pi$ for
    any great circle on $K/H$,
    \item and $\sqrt{g_t(v,v)}$ is for all $v\in\mathfrak{p}$ an even analytic function.
\end{enumerate}

The smoothness conditions which we have proven on the previous
pages will be sufficient for the purpose of this article. At the
end of this section, we study the question if the four-form
$\Omega$ has a smooth extension to the singular orbit. We assume
that we have already proven that the metric satisfies the
smoothness conditions. First, we consider the case where the
holonomy of the metric on the union of all principal orbits $M^0$
is $\text{Spin}(7)$. Since the metric is smooth, the holonomy of
$(M,g)$ equals $\text{Spin}(7)$, too. Therefore, there exists a
unique smooth $\text{Spin}(7)$-structure $\widetilde{\Omega}$ on
$M$, whose associated metric is $g$. $\Omega$ and
$\widetilde{\Omega}$ coincide on $M^0$. This observation proves
that $\widetilde{\Omega}$ is a smooth extension of $\Omega$ to the
singular orbit.

For our considerations, the case of holonomy $SU(4)$ is more important. Before we can prove the smoothness of $\Omega$ in this
case, we need the following lemma:

\begin{Le} \label{SU4Lemma}
\begin{enumerate}
    \item Let $M$ be an eight-dimensional manifold which carries a parallel $SU(4)$-structure $\mathfrak{G}$. We denote the
    space of all parallel $\text{Spin}(7)$-structures on $M$ which are an extension of $\mathfrak{G}$ and have the same
    extension to an $SO(8)$-structure as $\mathfrak{G}$ by $\mathcal{S}$. Any connected component of $\mathcal{S}$ is
    diffeomorphic to a circle.
    \item Let $M$ be an eight-dimensional manifold which carries a one-parame\-ter family $\mathcal{S}$ of parallel
    $\text{Spin}(7)$-structures. Moreover, let the extension of all $\text{Spin}(7)$-structures to an $SO(8)$-structure
    be the same and let $\mathcal{S}$ be diffeomorphic to a circle. Then, there also exists a parallel $SU(4)$-structure on $M$.
\end{enumerate}
\end{Le}

\begin{proof} Since all $G$-structures in the lemma are parallel,
it suffices to consider the situation at a single point. Moreover,
we can identify the tangent space of $M$ with $\mathbb{R}^8$ and
the groups $SU(4)$, $\text{Spin}(7)$, and $SO(8)$ with their real
eight-dimensional irreducible representations. We search for
matrices $A\in GL(8,\mathbb{R})$ such that conjugation with $A$
leaves $SU(4)$ and $SO(8)$ invariant but changes $\text{Spin}(7)$.
Since conjugation by a multiple of the identity matrix leaves any
group invariant, we can restrict ourselves to $A\in
SL(8,\mathbb{R})$. By algebraic arguments, we see that set of all
$A$ which leave $SU(4)$ and $SO(8)$ invariant is a group $G$ with
identity component $U(4)$. The subgroup of $G$ which also leaves
$\text{Spin}(7)$ invariant has $SU(4)$ as identity component. The
connected components of $\mathcal{S}$ thus are diffeomorphic to
$U(4)/SU(4) \cong S^1$.

In the situation of the second part of the lemma, the holonomy has
to be a proper subgroup of $\text{Spin}(7)$. All we have to prove
is that the holonomy is contained in $SU(4)$. The only case where
this is not true is where the holonomy equals $G_2$. If the
holonomy was $G_2$, there would be a parallel $G_2$-structure
$\omega$ and a parallel one-form $\alpha$ on $M$. Any parallel
$\text{Spin}(7)$-structure in $\mathcal{S}$ would be given by
$\ast\omega + \lambda\:\alpha\wedge\omega$ for a  $\lambda\in
\mathbb{R}\setminus\{0\}$. Since $\mathcal{S}$ has no subset which
is diffeomorphic to a circle, we have obtained a contradiction.
\end{proof}

\begin{Rem}
\label{SU4LemmaRemark} If we drop the word "parallel" and replace
the circle $\mathcal{S}$ by the sections of a certain circle
bundle over $M$, the statement of Lemma \ref{SU4Lemma} remains
true.
\end{Rem}

We now assume that the holonomy of $M^0$ is $SU(4)$ and that we
have checked the smoothness conditions for $g$. By the same
arguments as above, we can prove the existence of a unique smooth
$SU(4)$-structure on all of $M$. It follows from Lemma
\ref{SU4Lemma} that there exists a certain family $\mathcal{S}$ of
$\text{Spin}(7)$-structures on $M$. On $M^0$, $\Omega$ has to
coincide with one of them. Since all elements of $\mathcal{S}$ are
smooth at the singular orbit, $\Omega$ has a smooth extension to
the singular orbit. With help of the facts which we have collected
in this section we are now able to construct explicit examples of
$\text{Spin}(7)$-manifolds.

\section{Metrics with principal orbit $Q^{k,l,m}$}
\label{Q111}

In this section, we assume that the parallel
$\text{Spin}(7)$-structure is preserved by a cohomogeneity one
action of $SU(2)^3$. For our calculations, we will re\-present the
elements of $SU(2)^3$ by triples of complex $2\times 2$-matrices.
Since all principal orbits are seven-dimensional, the isotropy
group at a point of a principal orbit is isomorphic to $U(1)^2$.
There are infinitely many non-conjugate embeddings of $U(1)^2$
into $SU(2)^3$. We will describe them in detail and check if the
coset space $SU(2)^3/U(1)^2$ admits an $SU(2)^3$-invariant
$G_2$-structure. In order to do this, we first describe the
one-dimensional subalgebras of $3\mathfrak{su}(2)$. Up to
conjugation, any such subalgebra is embedded by a map $i_{k,l,m}
:\mathfrak{u}(1) \rightarrow 3\mathfrak{su}(2)$ with
$k,l,m\in\mathbb{Z}$ and

\begin{equation}
i_{k,l,m}(ix)
:=\left(\left(\begin{array}{cc}
ikx & 0 \\
0 & -ikx \\
\end{array}\right),
\left(\begin{array}{cc}
ilx & 0 \\
0 & -ilx \\
\end{array}\right),
\left(\begin{array}{cc}
imx & 0 \\
0 & -imx\\
\end{array}\right)\right)\:.
\end{equation}

Without loss of generality, we can assume that $(k,l,m)$ are coprime. Furthermore, we can even restrict ourselves to
non-negative values of $k$, $l$, and $m$: Let $\phi_P:SU(2)^3\rightarrow SU(2)^3$ be defined by $\phi_P(Q):=PQP^{-1}$, where

\begin{equation}
P:=\left(\left(\begin{array}{cc}
0 & 1 \\
-1 & 0 \\
\end{array}\right),
\left(\begin{array}{cc}
1 & 0 \\
0 & 1 \\
\end{array}\right),
\left(\begin{array}{cc}
1 & 0 \\
0 & 1\\
\end{array}\right)\right)\in SU(2)^3\:.
\end{equation}

$\phi_P$ maps $i_{k,l,m}(\mathfrak{u}(1))$ into
$i_{-k,l,m}(\mathfrak{u}(1))$. By another choice of $P$, we can
map $i_{k,l,m}(\mathfrak{u}(1))$ into
$i_{k,-l,m}(\mathfrak{u}(1))$ or $i_{k,l,-m}(\mathfrak{u}(1))$.
Therefore, our restriction to the case where $k,l,m\geq 0$ is
justified. It is also possible to choose $P$ as a permutation of
the three components of $(\mathbb{C}^2)^3$. Since the group which
is generated by the corresponding $\phi_P$ acts on $(k,l,m)$ by
permutations, too, we can finally assume that $k\geq l\geq m\geq
0$.

We continue describing the embeddings of $U(1)^2$ into $SU(2)^3$.
The one-dimen\-sion\-al subalgebras $i_{k,l,m}(\mathfrak{u}(1))$
together span a Cartan subalgebra of $3\mathfrak{su}(2)$, which we
denote throughout this section by $3\mathfrak{u}(1)$. The equation
$q(X,Y):=-\text{tr}(XY)$, where $X,Y\in 3\mathfrak{su}(2)$,
defines a biinvariant metric $q$ on $SU(2)^3$. The $q$-orthogonal
complement of $i_{k,l,m}(\mathfrak{u}(1))$ $\subseteq
3\mathfrak{u}(1)$ we denote by $2\mathfrak{u}(1)_{k,l,m}$ and the
Lie subgroup of $SU(2)^3$ whose Lie algebra is
$2\mathfrak{u}(1)_{k,l,m}$ by $U(1)^2_{k,l,m}$. The quotient
$SU(2)^3/U(1)^2_{k,l,m}$ is called $Q^{k,l,m}$. Since $U(1)^2$ can
be mapped by a conjugation inside any maximal torus of $SU(2)^3$,
any quotient $SU(2)^3/U(1)^2$ is $SU(2)^3$-equivariantly
diffeomorphic to a $Q^{k,l,m}$.

Our next step is to choose a basis $(e_1,\ldots,e_9)$ of $3\mathfrak{su}(2)$. In order to do this, we first define

\begin{equation}
\sigma_1:=\frac{1}{2}\left(\,\begin{array}{cc}
0 & i\\
i & 0 \\
\end{array}\,\right)\quad
\sigma_2:=\frac{1}{2}\left(\,\begin{array}{cc}
0 & 1 \\
-1 & 0 \\
\end{array}\,\right)\quad
\sigma_3:=\frac{1}{2}\left(\,\begin{array}{cc}
i & 0 \\
0 & -i \\
\end{array}\,\right)\:.
\end{equation}

With this notation, we further define

\begin{equation}
\begin{array}{ll}
e_1:=(\sigma_1,0,0) & e_2:=(\sigma_2,0,0)\\
e_3:=(0,\sigma_1,0) & e_4:=(0,\sigma_2,0)\\
e_5:=(0,0,\sigma_1) & e_6:=(0,0,\sigma_2) \\
e_7:=(k\sigma_3,l\sigma_3,m\sigma_3) & e_8:=(l\sigma_3,-k\sigma_3,0)\\
e_9:=(mk\sigma_3,ml\sigma_3,-(k^2+l^2)\sigma_3) & \\
\end{array}
\end{equation}

The tangent space of $Q^{k,l,m}$ can be identified with the
$q$-orthogonal complement $\mathfrak{m}$ of
$2\mathfrak{u}(1)_{k,l,m}$ in $3\mathfrak{su}(2)$. $\mathfrak{m}$
is spanned by $(e_1,\ldots,e_7)$ and $2\mathfrak{u}(1)_{k,l,m}$ is
spanned by $(e_8,e_9)$. $e_8$ and $e_9$ act on the tangent space
by the isotropy representation of $2\mathfrak{u}(1)_{k,l,m}$. This
action can be described by the commutator of matrices from
$2\mathfrak{u}(1)_{k,l,m}$ and $\mathfrak{m}$. We recall that we
want $Q^{k,l,m}$ to admit an $SU(2)^3$-invariant $G_2$-structure.
According to Lemma \ref{G2HomLemma}, this is the case if and only
if the isotropy representation of $U(1)^2_{k,l,m}$ is equivalent
to the action of a Cartan subalgebra of $G_2$ on
$\text{Im}(\mathbb{O})$. Since $U(1)^2_{k,l,m}$ is connected, it
suffices to consider the action of its Lie algebra. The isotropy
action of $2\mathfrak{u}(1)_{k,l,m}$ on $\mathfrak{m}$ yields the
following subalgebra of $\mathfrak{gl}(\mathfrak{m})$:

\begin{equation}
\left\{\left(\,\begin{array}{ccccccc} \hhline{--~~~~~}
\multicolumn{1}{|c}{0} & \multicolumn{1}{c|}{x} &&&&&\\
\multicolumn{1}{|c}{-x} & \multicolumn{1}{c|}{0} &&&&&\\
\hhline{----~~~} && \multicolumn{1}{|c}{0} &
\multicolumn{1}{c|}{y} &&&\\ && \multicolumn{1}{|c}{-y} &
\multicolumn{1}{c|}{0} &&&\\ \hhline{~~----~} &&&&
\multicolumn{1}{|c}{0} & \multicolumn{1}{c|}{z} &\\ &&&&
\multicolumn{1}{|c}{-z} & \multicolumn{1}{c|}{0} &\\
\hhline{~~~~---} &&&&&& \multicolumn{1}{|c|}{0}\\ \hhline{~~~~~~-}
\end{array}\,\right)\right.
\left|
\begin{array}{l} \\ \\ \\ \\ \\ \\ \\ \\ \end{array}
kx + ly + mz = 0 \right\}\:,
\end{equation}

where the matrix representation is with respect to the basis
$(e_1,\ldots,e_7)$. We equip $\mathbb{R}^7$ with the action of the
Cartan subalgebra

\begin{equation}
\label{CSAG2} \left\{\left(\,\begin{array}{ccccccc}
\hhline{-~~~~~~} \multicolumn{1}{|c|}{0} &&&&&&\\ \hhline{---~~~~}
& \multicolumn{1}{|c}{0} & \multicolumn{1}{c|}{x} &&&&\\ &
\multicolumn{1}{|c}{-x} & \multicolumn{1}{c|}{0} &&&&\\
\hhline{~----~~} &&& \multicolumn{1}{|c}{0} &
\multicolumn{1}{c|}{y} &&\\ &&& \multicolumn{1}{|c}{-y} &
\multicolumn{1}{c|}{0} &&\\ \hhline{~~~----} &&&&&
\multicolumn{1}{|c}{0} & \multicolumn{1}{c|}{z} \\ &&&&&
\multicolumn{1}{|c}{-z} & \multicolumn{1}{c|}{0} \\
\hhline{~~~~~--}
\end{array}\,\right)\right.\left|
\begin{array}{l} \\ \\ \\ \\ \\ \\ \\ \\ \end{array}
x + y - z = 0 \right\}
\end{equation}

of $\mathfrak{g}_2$. By comparing the weights of both actions, we
see that $Q^{k,l,m}$ admits an $SU(2)^3$-invariant $G_2$-structure
if and only if $\{k,l,m\}\in\{-1,1\}$. Without loss of generality,
we assume that $k=l=m=1$.

Our next task is to describe the set of all $SU(2)^3$-invariant
cosymplectic $G_2$-structures on $Q^{1,1,1}$. We split this
problem into three subproblems: First, we classify all
$SU(2)^3$-invariant metrics on $Q^{1,1,1}$, second we classify all
invariant $G_2$-structures whose associated metric and orientation
are fixed ones, and third, we prove which of them are
cosymplectic. For any $SU(2)^3$-invariant metric $g$ on
$Q^{1,1,1}$ there exists a $U(1)^2_{1,1,1}$-equivariant,
$q$-symmetric, positive definite endomorphism $\varphi$ of
$\mathfrak{m}$, which is defined by $q(\varphi(X),Y)=g(X,Y)$.
Conversely, any such $\varphi$ yields an $SU(2)^3$-invariant
metric on $Q^{1,1,1}$. $\mathfrak{m}$ splits into the following
irreducible $2\mathfrak{u}(1)_{1,1,1}$-submodules:

\begin{equation}
\begin{array}{llll}
V_1 & :=\text{span}(e_1,e_2) & V_2 & :=\text{span}(e_3,e_4)\\
V_3 & :=\text{span}(e_5,e_6) & V_4 & :=\text{span}(e_7)\\
\end{array}
\end{equation}

In order to describe the invariant metrics, we have to check if
any pair of the above $2\mathfrak{u}(1)_{1,1,1}$-modules is
equivalent. It is easy to see that $V_1$, $V_2$, and $V_3$ are
pairwise inequivalent, since on any pair of those spaces either
the one-dimensional Lie algebra generated by $e_8$ or the Lie
algebra generated by $e_9$ acts with different weights. $V_4$
cannot be equivalent to one of the other modules, since it is of
lower dimension. We conclude with help of Schur's lemma that the
$SU(2)^3$-invariant metrics $g$ on $Q^{1,1,1}$ are precisely those
which satisfy

\begin{equation}
\label{Q111GenMetric}
\begin{split}
g = & a^2 (e^1\otimes e^1 + e^2\otimes e^2) + b^2(e^3\otimes e^3 + e^4\otimes e^4)\\
\quad & + c^2(e^5\otimes e^5 + e^6\otimes e^6) + f^2 e^7\otimes e^7\\
\end{split}
\end{equation}

with $a,b,c,f\in\mathbb{R}\setminus\{0\}$. The most generic
$G_2$-structure on $Q^{1,1,1}$ can be described conveniently by a
three-form whose coefficients are odd powers of $a$, $b$, $c$, and
$f$. Therefore, we allow those parameters to take negative values
although this does not change the metric. If $g$ is a
cohomogeneity one metric with principal orbit $Q^{1,1,1}$, $a$,
$b$, $c$, and $f$ turn into functions which are defined on the
interval $M/SU(3)$.

Our next step is to describe the set of all homogeneous
$G_2$-structures on $Q^{1,1,1}$ whose extension to an
$SO(7)$-structure is a fixed one. On the following pages,
$2\mathfrak{u}(1)$ will denote the Cartan subalgebra (\ref{CSAG2})
of $\mathfrak{g}_2$. The maximal torus of $G_2$ whose Lie algebra
is $2\mathfrak{u}(1)$ we will denote by $U(1)^2$. As we have
proven in Lemma \ref{G2StrucSpace}, the set we search for is

\begin{equation}
\text{Norm}_{SO(7)}U(1)^2/\text{Norm}_{G_2}U(1)^2\:.
\end{equation}

First, we describe the numerator of the above quotient. The Lie algebra of $\text{Norm}_{SO(7)}U(1)^2$ is

\begin{equation}
\text{Norm}_{\mathfrak{so}(7)}2\mathfrak{u}(1):=\{x\in\mathfrak{so}(7)|
\text{ad}_x(2\mathfrak{u}(1))\subseteq 2\mathfrak{u}(1)\}\:.
\end{equation}

From now on $3\mathfrak{u}(1)$ denotes the Cartan subalgebra

\begin{equation}
\label{CSASO7} \left\{\left(\,\begin{array}{ccccccc}
\hhline{-~~~~~~} \multicolumn{1}{|c|}{0} &&&&&& \\
\hhline{---~~~~} & \multicolumn{1}{|c}{0} & \multicolumn{1}{c|}{a}
&&&&\\ & \multicolumn{1}{|c}{-a} & \multicolumn{1}{c|}{0} &&&&\\
\hhline{~----~~} &&& \multicolumn{1}{|c}{0} &
\multicolumn{1}{c|}{b} &&\\ &&& \multicolumn{1}{|c}{-b} &
\multicolumn{1}{c|}{0} &&\\ \hhline{~~~----} &&&&&
\multicolumn{1}{|c}{0} & \multicolumn{1}{c|}{c} \\ &&&&&
\multicolumn{1}{|c}{-c} & \multicolumn{1}{c|}{0} \\
\hhline{~~~~~--}
\end{array}\,\right)\right.\left|
\begin{array}{c} \\ \\ \\ \\ \\ \\ \\ \\ \end{array}
a,b,c\in\mathbb{R}\right\}
\end{equation}

of $\mathfrak{so}(7)$. Obviously,
$3\mathfrak{u}(1)\subseteq\text{Norm}_{\mathfrak{so}(7)}
2\mathfrak{u}(1)$. We will prove that the other inclusion is also
satisfied. For any $x\in\text{Norm}_{\mathfrak{so}(7)}
2\mathfrak{u}(1)$ we have

\begin{equation}
[x,z]\in 2\mathfrak{u}(1)\quad\forall z\in 2\mathfrak{u}(1)\:.
\end{equation}

Since the Killing form $\kappa$ of $\mathfrak{so}(7)$ is
associative, it follows that for any $y\in 2\mathfrak{u}(1)$

\begin{equation}
\kappa([x,z],y) = \kappa(x,[z,y]) = 0
\end{equation}

and therefore $[x,z]=0$. This observation proves that the
normalizer equals the centralizer

\begin{equation}
C_{\mathfrak{so}(7)}\:2\mathfrak{u}(1):=\{x\in\mathfrak{so}(7)|\text{ad}_x(
2\mathfrak{u}(1)) = \{0\}\}\:.
\end{equation}

For the following considerations, we complexify the Lie algebras $\mathfrak{so}(7)$ and $2\mathfrak{u}(1)$ and return to the
real case later on. $x$ has a Cartan decomposition

\begin{equation}
\label{CartanDecomp}
x=x_h+\sum\limits_{\alpha\in\Phi}\mu_{\alpha}g_{\alpha}\quad\text{with}\quad
x_h\in 3\mathfrak{u}(1) \otimes \mathbb{C}\:,\mu_{\alpha}\in
\mathbb{C} \:,g_{\alpha}\in L_{\alpha}\:,
\end{equation}

where $\Phi$ is the root system of $\mathfrak{so}(7,\mathbb{C})$ and $L_{\alpha}$ is the eigenspace to the eigenvalue
$\alpha:3\mathfrak{u}(1)\otimes \mathbb{C}\rightarrow\mathbb{C}$ of $\text{ad}_{3\mathfrak{u}(1)\otimes\mathbb{C}}$. Let $z\in
2\mathfrak{u}(1)\otimes\mathbb{C}$ be arbitrary. Applying $\text{ad}_z$ to (\ref{CartanDecomp}) yields the following equation:

\begin{equation}
\text{ad}_z(x) = \sum\limits_{\alpha\in\Phi}
\mu_{\alpha}\alpha(z)g_{\alpha}\:.
\end{equation}

We want to prove that $x$ cannot be a non-zero element of the
orthogonal complement of $3\mathfrak{u}(1)\otimes\mathbb{C}$. If
there exists an $\alpha\in\Phi$ with $\alpha(2\mathfrak{u}(1)
\otimes \mathbb{C})=0$, then $[z,g_{\alpha}]=\alpha(z)g_{\alpha}
=0$ for all $z\in 2\mathfrak{u}(1)\otimes \mathbb{C}$. In that
situation, we could choose $x$ as $g_{\alpha}$. Conversely, we
assume that there is no such $\alpha$ and choose
$(3\mathfrak{u}(1) \otimes \mathbb{C})^\perp \setminus\{0\}\ni x=
\sum_{\alpha\in\Phi}\mu_{\alpha}g_{\alpha}$ arbitrarily. If we
choose a $z$ with $\alpha(z)\neq 0$ for an $\alpha$ with
$\mu_{\alpha}\neq 0$, we have $\text{ad}_z(x)\neq 0$. In that
case, $\text{Norm}_{\mathfrak{so}(7,\mathbb{C})}2\mathfrak{u}(1)
\otimes \mathbb{C}$ would be $3\mathfrak{u}(1)\otimes\mathbb{C}$.
We therefore have to answer the question if there is an $\alpha$
with $\alpha(2\mathfrak{u}(1) \otimes \mathbb{C})=0$.

In order to do this, we have to take a closer look at the root
system of $\mathfrak{so}(7,\mathbb{C})$. Let $L_1$ be the element
of the Cartan subalgebra (\ref{CSASO7}) with $a=1$ and $b=c=0$.
Analogously, let $L_2$ be given by $b=1$, $a=c=0$, and $L_3$ by
$c=1$ and $a=b=0$. We denote the dual of $L_j$ with respect to the
Killing form by $\theta_j$. The root system of
$\mathfrak{so}(7,\mathbb{C})$ is:

\begin{equation}
\{\pm \theta_j|1\leq j \leq 3\}\cup\{\pm \theta_j \pm
\theta_k|1\leq j<k\leq 3 \}\:.
\end{equation}

The Cartan subalgebra $2\mathfrak{u}(1) \otimes \mathbb{C}$ of $\mathfrak{g}_2^{\mathbb{C}}$ is the plane which is orthogonal
to $L_1+L_2-L_3$. The $\alpha\in\Phi$ which vanish on $2\mathfrak{u}(1) \otimes \mathbb{C}$ are precisely those which are
multiples of $\theta_1+\theta_2-\theta_3$. Since there is no root of $\mathfrak{so}(7,\mathbb{C})$ with this property, we have
proven that indeed $\text{Norm}_{\mathfrak{so}(7,\mathbb{C})}(2\mathfrak{u}(1) \otimes \mathbb{C})=3\mathfrak{u}(1) \otimes
\mathbb{C}$. By passing to the compact real form of $\mathfrak{so}(7,\mathbb{C})$, we can conclude that
$\text{Norm}_{\mathfrak{so}(7)} 2\mathfrak{u}(1) = 3\mathfrak{u}(1)$.

Let $U(1)^3$ be the maximal torus of $SO(7)$ with Lie algebra $3\mathfrak{u}(1)$. Our next step is to describe the discrete
group $\Gamma:=(\text{Norm}_{SO(7)}  U(1)^2)/U(1)^3$. The group $(\text{Norm}_{SO(7)}U(1)^3)/U(1)^3$ is isomorphic to the Weyl
group $\mathcal{W}_{\mathfrak{so(7)}}$ of $\mathfrak{so}(7)$. We prove that $\text{Norm}_{SO(7)}U(1)^2 \subseteq
\text{Norm}_{SO(7)}U(1)^3$. Let us assume that there is an element $h$ of $SO(7)$ such that $\text{Ad}_h$ leaves
$2\mathfrak{u}(1)$ invariant, but does not leave $3\mathfrak{u}(1)$ invariant. Then $3\mathfrak{u}(1)$ and
$\text{Ad}_h(3\mathfrak{u}(1))$ are two distinct Cartan subalgebras whose intersection is $2\mathfrak{u}(1)$. In this
situation, the centralizer of $2\mathfrak{u}(1)$ is at least four-dimensional which is not the case. $\Gamma$ therefore is a
subgroup of $\mathcal{W}_{\mathfrak{so(7)}}$. More precisely, it is the subgroup of $\mathcal{W}_{\mathfrak{so(7)}}$ which
leaves the plane $2\mathfrak{u}(1)\subseteq 3\mathfrak{u}(1)$ invariant.

In order to describe $\Gamma$ explicitly, we introduce some facts on $\mathcal{W}_{\mathfrak{so(7)}}$. The Weyl group of
$\mathfrak{so}(7)$ is isomorphic to $\mathbb{Z}_2^3\rtimes S_3$, which is of order $48$. The first factor of
$\mathcal{W}_{\mathfrak{so}(7)}$ acts by changing the signs of the $\theta_i$. The second factor of the Weyl group consists of
the permutations of $\{\theta_1,\theta_2,\theta_3\}$. $2\mathfrak{u}(1)$ is the plane of all $x\in 3\mathfrak{u}(1)$ satisfying:

\begin{equation}
\theta_1(x)+\theta_2(x)-\theta_3(x)=0\:.
\end{equation}

By replacing $L_3$ by $-L_3$, we can change this equation into:

\begin{equation}
\theta_1(x)+\theta_2(x)+\theta_3(x)=0\:.
\end{equation}

The subgroup of $\mathcal{W}_{\mathfrak{so}(7)}$ which leaves this
equation invariant is generated by the permutations and the
simultaneous change of all signs. $\Gamma$ thus is the direct
product $\mathbb{Z}_2\times S_3$, which is isomorphic to the
Dieder group $D_6$. All in all, we have proven:

\begin{equation}
\text{Norm}_{SO(7)}U(1)^2=U(1)^3\rtimes D_6\:.
\end{equation}

Next, we have to determine $\text{Norm}_{G_2}U(1)^2$. Since
$2\mathfrak{u}(1)$ is a Cartan subalgebra of $\mathfrak{g}_2$,
$\text{Norm}_{G_2} U(1)^2 / U(1)^2$ is the Weyl group
$\mathcal{W}_{\mathfrak{g}_2}$ of $\mathfrak{g}_2$. It is known
that $\mathcal{W}_{\mathfrak{g}_2}$ is isomorphic to $D_6$, too.
$\text{Norm}_{G_2}U(1)^2$ therefore is a semidirect product
$U(1)^2\rtimes D_6$. There is the following exact sequence:

\begin{equation}
\begin{split}
\pi_0(\text{Norm}_{G_2}U(1)^2) & \overset{\pi_0(i)}{\longrightarrow} \pi_0(\text{Norm}_{SO(7)}U(1)^2) \\
& \overset{\pi_0(\pi)}{\longrightarrow} \pi_0( \text{Norm}_{SO(7)}U(1)^2/\text{Norm}_{G_2} U(1)^2)\rightarrow\{0\}\:,
\end{split}
\end{equation}

where $\pi_0(i)$ is induced by the inclusion of
$\text{Norm}_{G_2}U(1)^2$ into $\text{Norm}_{SO(7)}U(1)^2$ and
$\pi_0(\pi)$ by the projection map. $\text{Norm}_{G_2}U(1)^2$ and
$\text{Norm}_{SO(7)}U(1)^2$ have both 12 connected components. If
we were able to prove that $\pi_0(i)$ is surjective, we could
conclude that $\text{Norm}_{SO(7)}U(1)^2/\text{Norm}_{G_2} U(1)^2$
is connected. Since $\pi_0(\text{Norm}_{SO(7)}U(1)^2)$ and
$\pi_0(\text{Norm}_{G_2}U(1)^2)$ are both finite, we can prove the
injectivity instead. Let $x\in\text{Norm}_{G_2}U(1)^2$ with
$x\notin U(1)^2$ be arbitrary. Then $x=\alpha x_0$ with
$\alpha\in\mathcal{W}_{\mathfrak{g}_2}\setminus\{e\}$ and $x_0\in
U(1)^2$. Since $\alpha$ acts non-trivially on the dual of the
Cartan subalgebra $2\mathfrak{u}(1)$, it cannot be an element of
the maximal torus $U(1)^3$ of $SO(7)$. Therefore, $x$ is not an
element of the identity component of $\text{Norm}_{SO(7)}U(1)^2$
and $\pi_0(i)$ is thus injective. All in all, we have proven that
$\text{Norm}_{SO(7)}U(1)^2 / \text{Norm}_{G_2}U(1)^2$ is
connected. More precisely, it is a group which is isomorphic to
$U(1)$.

For our considerations we need to describe the action of
$\text{Norm}_{SO(7)}U(1)^2/$ $\text{Norm}_{G_2}U(1)^2$ on
$\mathfrak{m}$. We define

\begin{equation}
T:=\left\{\left(\,\begin{array}{cccc} \hhline{-~~~}
\multicolumn{1}{|c|}{R_{\theta}} &&&\\ \hhline{--~~} &
\multicolumn{1}{|c|}{R_{\theta}} && \\ \hhline{~--~} &&
\multicolumn{1}{|c|}{R_{\theta}} & \\ \hhline{~~--} &&&
\multicolumn{1}{|c|}{1} \\ \hhline{~~~-}
\end{array}\,\right)=: T_{\theta}\right.\:
\left|
\begin{array}{c} \\ \\ \\ \\ \end{array}\!
\theta\in\mathbb{R}\right\} \:,
\end{equation}

where $R_{\theta}$ denotes the rotation in the plane around the
angle $\theta$ and the matrix representation of $T_{\theta}$ is
with respect to the basis $(e_1,\ldots,e_7)$ of $\mathfrak{m}$.
$T$ commutes with the action of $U(1)^2_{1,1,1}$ and the
intersection $T\cap U(1)^2_{1,1,1}$ is discrete. Furthermore,
$T_{\theta}$ is orthogonal with respect to $g$ and orientation
preserving. Since $U(1)^2_{1,1,1}$ preserves any
$SU(2)^3$-invariant $G_2$-structure $\omega$ and $\text{rank}\:
G_2=2$, the action of $T$ cannot preserve $\omega$. All in all,
the set of all invariant $G_2$-structures with the same associated
metric and orientation as $\omega$ is generated by $T$. Let $S$ be
the following subgroup of $\text{Norm}_{SU(2)^3}U(1)^2$:

\begin{equation}
\label{DefSQ111}
\left\{\left(\left(
\begin{array}{cc}
e^{i\tfrac{\theta}{2}} & 0 \\
0 & e^{-i\tfrac{\theta}{2}} \\
\end{array}\right), \left(
\begin{array}{cc}
e^{i\tfrac{\theta}{2}} & 0 \\
0 & e^{-i\tfrac{\theta}{2}} \\
\end{array}\right), \left(
\begin{array}{cc}
e^{i\tfrac{\theta}{2}} & 0 \\
0 & e^{-i\tfrac{\theta}{2}} \\
\end{array}\right)\right)=:S_{\theta}\right.\:
\left|
\begin{array}{c} \\ \\ \\ \end{array}\!
\theta\in\mathbb{R}\right\}
\end{equation}

Conjugation by $S_{\theta}$ is a well-defined diffeomorphism of
$Q^{1,1,1}$. Moreover, it is an orientation preserving isometry.
Its differential acts as $T_{\theta}$ on $\mathfrak{m}$. Our set
of $G_2$-structures can therefore be obtained by an isometric
action of $S$ on $\omega$. The simultaneous action of an
$S_{\theta}$ on all principal orbits can be extended to an
isometry $\Phi_{\theta}$ of the cohomogeneity one manifold $M$. We
assume that $d\ast\omega=0$ and denote the extension of $\omega$
to a parallel $\text{Spin}(7)$-structure by $\Omega$. Since
$\Phi^\ast_{\theta}\Omega$ is parallel if $\Omega$ is parallel, it
suffices to consider only one cosymplectic $\omega$ on the
principal orbit instead of the whole one-parameter family
generated by $S$.

We fix an arbitrary $2\mathfrak{u}(1)_{1,1,1}$-invariant metric
$g$ on $\mathfrak{m}$, which has to be of type
(\ref{Q111GenMetric}). Furthermore, we assume that
$g(\tfrac{\partial}{\partial t},\tfrac{\partial}{\partial t})=1$.
Our aim is to construct a basis $(f_i)_{0\leq i\leq 7}$ of the
tangent space
$\mathfrak{m}\oplus\text{span}(\tfrac{\partial}{\partial t})$
which yields an $SU(2)^3$-invariant $\text{Spin}(7)$-structure
with $g$ as associated metric. We can change the orientation of
$M$ by replacing $t$ by $-t$. Therefore, we do not have to take
care of orientation issues. Let $\psi: \text{Im}(\mathbb{O})
\rightarrow \mathfrak{m}$ be the isomorphism which maps the
$G_2$-structure on $\text{Im}(\mathbb{O})$ to the tangent space of
$Q^{1,1,1}$. $\psi$ has to preserve the inner product of both
spaces. Moreover, it has to turn the Cartan subalgebra
(\ref{CSAG2}) into the isotropy representation of
$2\mathfrak{u}(1)_{1,1,1}$. If all these conditions are satisfied,
$(f_i)_{1\leq i\leq 7}$ can be chosen as
$(\psi(i),\ldots,\psi(k\epsilon))$. By choosing
$f_0:=\tfrac{\partial}{\partial t}$, we obtain a cohomogeneity one
$\text{Spin}(7)$-structure. A possible basis of this kind is:

\begin{equation}
\label{CBasisQ111}
\begin{array}{llll}
f_0:=\frac{\partial}{\partial t} & f_1:=\frac{1}{f}e_7 & f_2:=\frac{1}{a}e_1 & f_3:=\frac{1}{a}e_2\\
&&&\\
f_4:=\frac{1}{b}e_3 & f_5:=\frac{1}{b}e_4 & f_6:=\frac{1}{c}e_6 & f_7:=\frac{1}{c}e_5\\
\end{array}
\end{equation}

Since the group $S$ acts isometrically and transitively on the set
of all $G_2$-structures with a fixed extension to an
$SO(7)$-structure, this is the most general basis which we have to
consider. The basis (\ref{CBasisQ111}) yields the following
$\text{Spin}(7)$-structure:

\begin{equation}
\begin{split}
\Omega & =
abcf\:e^{1357}-abcf\:e^{1467}-abcf\:e^{2367}-abcf\:e^{2457}\\
&\quad -a^2b^2\:e^{1234}-a^2c^2\:e^{1256}-b^2c^2\:e^{3456}\\
&\quad -a^2f\:e^{127}\wedge dt-b^2f\:e^{347}\wedge dt
-c^2f\:e^{567}\wedge dt\\ &\quad -abc\:e^{136}\wedge
dt-abc\:e^{145}\wedge dt-abc\:e^{235}\wedge dt +abc\:e^{246}\wedge
dt\\
\end{split}
\end{equation}

Let $X$ and $Y$ be left-invariant vector fields on $SU(2)^3$ and
$\alpha$ be a left-invariant one-form. We have $d\alpha(X,Y) =
-\alpha([X,Y])$. With help of the anti-derivation property of $d$
we can compute the exterior derivatives of the pull-backs
$\pi^\ast\omega$ and $\pi^\ast\ast\omega$ of $\omega$ and
$\ast\omega$ to $SU(2)^3$. This enables us to calculate $d\Omega$.
We finally see that $d\Omega=0$ is equivalent to:

\begin{equation}
\label{HolRedQ111}
\begin{split}
a' & =-\frac{1}{6}\frac{f}{a}\\ b' & =-\frac{1}{6}\frac{f}{b}\\ c'
& =-\frac{1}{6}\frac{f}{c}\\ f' &
=\frac{1}{6}\frac{f^2}{a^2}+\frac{1}{6}\frac{f^2}{b^2}
+\frac{1}{6}\frac{f^2}{c^2}-3\\
\end{split}
\end{equation}

Our next step is to solve the above system. The initial conditions for (\ref{HolRedQ111}) shall be $a(0)=a_0$, $b(0)=b_0$,
$c(0)=c_0$, and $f(0)=f_0$. Since we have:

\begin{equation}
a'a=b'b=c'c=-\frac{1}{6}f\:,
\end{equation}

the functions $a^2$, $b^2$, and $c^2$ are up to a constant summand the same. We introduce a function $F$ with $F'=f$. By
requiring $F(0)=0$, we make $F$ unique. From the above equation, it follows that

\begin{equation}
\label{SolutionfQ1111}
a^2-a_0^2=b^2-b_0^2=c^2-c_0^2=-\frac{1}{3}F\:.
\end{equation}

We rewrite the fourth equation of (\ref{HolRedQ111}):

\begin{equation}
f' =-\frac{1}{6}\frac{f^2}{\frac{1}{3}F-a_0^2}
-\frac{1}{6}\frac{f^2}{\frac{1}{3}F-b_0^2}
-\frac{1}{6}\frac{f^2}{\frac{1}{3}F-c_0^2}-3\:.\\
\end{equation}

Since there is only one principal orbit, $f$ does not change its sign and $F$ is injective. The function $\widetilde{f}$ with
$\widetilde{f}(t) := f(F^{-1}(t))$ satisfies the equation:

\begin{equation}
\widetilde{f}\widetilde{f}' = -\frac{\widetilde{f}^2}{2t-6a_0^2} - \frac{\widetilde{f}^2}{2t-6b_0^2} -\frac{
\widetilde{f}^2}{2t-6c_0^2}-3\:.
\end{equation}

This equation is linear in $f^2$ and can be solved explicitly. By variation of constants, we obtain the following solution of
our initial value problem:

\begin{equation}
\label{SolutionfQ1112}
\begin{split}
\widetilde{f}(t)^2 &  =
\frac{f_0^2}{(1-\tfrac{t}{3a_0^2})(1-\tfrac{t}{3b_0^2})(
1-\tfrac{t}{3c_0^2})}\\ &
-\frac{3}{(t-3a_0^2)(t-3b_0^2)(t-3c_0^2)}\int_0^t
(s-3a_0^2)(s-3b_0^2)(s-3c_0^2)ds\:.\\
\end{split}
\end{equation}

The equations (\ref{SolutionfQ1111}) and (\ref{SolutionfQ1112})
describe the metric completely. Since we are in the lucky
situation to have explicit solutions of the equations of the
holonomy reduction, we are able to describe the global shape of
the metric. The function $\widetilde{f}$ is of type $c\sqrt{t} +
O(1)$ for large values of $t$ and a
$c\in\mathbb{R}\setminus\{0\}$. We insert the definition of
$\widetilde{f}$ and obtain:


\begin{equation}
f(t) = c\sqrt{F(t)} + O(1)\:.
\end{equation}

From this equation we can deduce that $f(t)=\tfrac{c^2}{2}t +
O(1)$ and from (\ref{SolutionfQ1111}) it follows that $a$, $b$,
and $c$ approach linear functions, too. This proves that the metric is
asymptotically conical and in particular complete. We make the
ansatz $a(t)=a_1t$, $\ldots$, $f(t)=f_1t$ and obtain
$a_1^2,b_1^2,c_1^2=\tfrac{1}{8}$ and $f_1=\tfrac{3}{4}$. These
numbers determine the metric on the base of the cone which our
cohomogeneity one metric approaches. Since the cone has holonomy
$\text{Spin}(7)$, $(a_1,b_1,c_1,f_1)$ describes a nearly parallel
$G_2$-structure on $Q^{1,1,1}$.


We now determine the holonomy of our metrics. The diffeomorphisms $\Phi^\ast_{\theta}$ preserve the metric and orientation but
not the $\text{Spin}(7)$-structure. The set of all $\text{Spin}(7)$-structures which we obtain by the action of $S$ is
diffeomorphic to a circle. According to Lemma \ref{SU4Lemma}, the holonomy is contained in $SU(4)$. If the holonomy is not all
of $SU(4)$, it is either a subgroup of $SU(3)$ or of $Sp(2)$. In the first case, there exists a parallel vector field $X$ on
$M$. The holonomy bundle is invariant under the isometry group of $M$. We can therefore assume that $X$ is
$SU(2)^3$-invariant. This is the case if and only if $X=c_1 \tfrac{1}{f} e_7 + c_2\tfrac{\partial}{\partial t}$, where $c_1$
and $c_2$ depend on $t$ only. Since the length of $X$ is constant and $\nabla_{\tfrac{\partial}{\partial t}}X=0$, $c_1$ and
$c_2$ have to be constant, too. If $c_1\neq 0$, $f$ is also a constant non-zero function. It follows from
(\ref{SolutionfQ1111}), that $a^2$, $b^2$, and $c^2$ are either all strictly increasing or strictly decreasing. In any
case, the right hand side of the fourth equation of (\ref{HolRedQ111}) cannot vanish. This is a contradiction to $f$ being
constant. If $X$ was a multiple of $\tfrac{\partial}{\partial t}$, we would have $a'=b'=c'=f'=0$, which is impossible, too.

If the holonomy was a subgroup of $Sp(2)$, there would exist three linearly independent K\"ahler forms on $M$. For similar
reasons as above, any K\"ahler form $\eta$ on $M$ has to be $SU(2)^3$-invariant. The two-forms

\begin{equation}
e^{12}\:,\quad e^{34}\:,\quad e^{56}\:,\quad e^7\wedge dt
\end{equation}

span the space of all $SU(2)^3$-invariant two-forms on $M$. The K\"ahler-form thus has to satisfy:

\begin{equation}
\eta = \epsilon_1 a^2 e^{12} + \epsilon_2 b^2 e^{34} + \epsilon_3
c^2 e^{56} + \epsilon_4 f e^7\wedge dt\quad\text{with}\quad
\epsilon_1,\epsilon_2,\epsilon_3,\epsilon_4\in\{-1,1\} \:.
\end{equation}

For the exterior derivative of $\eta$, we obtain:

\begin{equation}
\begin{split}
d\eta & =\frac{1}{3}\epsilon_4 f\:e^{12}\wedge dt+\epsilon_1 2a'a\:e^{12}\wedge dt + \frac{1}{3}\epsilon_4 f\:e^{34}\wedge dt
+\epsilon_2 2b'b\:e^{34}\wedge dt\\ &\quad +\frac{1}{3}\epsilon_4 f\:e^{56}\wedge dt+\epsilon_3 2c'c\:e^{56}\wedge dt\:.\\
\end{split}
\end{equation}

$d\eta=0$ yields the following equations:

\begin{equation}
\label{EtaEquationsQ111} \epsilon_4 f=-6\epsilon_1 a'a=-6\epsilon_2
b'b=-6\epsilon_3 c'c\:.
\end{equation}

By comparing (\ref{EtaEquationsQ111}) with the system (\ref{HolRedQ111}), we see that $\epsilon_1 = \epsilon_2 = \epsilon_3 =
\epsilon_4 = \pm 1$ and thus:

\begin{equation}
\label{EtaQ111} \eta = \pm (a^2 e^{12} + b^2 e^{34} + c^2 e^{56} + f
e^7\wedge dt)\:.
\end{equation}

Since there are only two linearly dependent K\"ahler forms on $M$, the holonomy is all of $SU(4)$. We finally investigate the
behaviour of the metric near the singular orbit. The possible singular orbits of a cohomogeneity one manifold with principal
orbit $Q^{1,1,1}$ are classified by the following lemma:

\begin{Le} \label{SingOrbQ111} Let $U(1)^2_{1,1,1}\subseteq SU(2)^3$ be chosen as at the beginning of this section. For reasons
of convenience, we denote this subgroup by $U(1)^2$ and its Lie algebra by $2\mathfrak{u}(1)$. Furthermore, let $K$ with
$U(1)^2\subsetneq K\subseteq SU(2)^3$ be a closed, connected subgroup. We denote the Lie algebra of $K$ by $\mathfrak{k}$. In
this situation, $\mathfrak{k}$ and $K$ can be found in the table below. Furthermore, $K/U(1)^2$ and $SU(2)^3/K$ satisfy the
following topological conditions:

\begin{center}
\begin{tabular}{|l|l|l|l|}
\hline

$\mathfrak{k}$ & $K$ & $K/U(1)^2$ & $SU(2)^3/K$\\

\hline \hline

$3\mathfrak{u}(1)$ & $U(1)^3$ & $\cong S^1$ & $\cong S^2\times
S^2\times S^2$\\

\hline

$2\mathfrak{u}(1)\oplus \mathfrak{su}(2)$ & $U(1)^2\times SU(2)$ &
$\cong S^3$ & $\cong S^2\times S^2$\\

\hline

$\mathfrak{u}(1)\oplus 2\mathfrak{su}(2)$ & $U(1)\times SU(2)^2$ &
$\not\cong S^5/\Gamma$ & $\cong S^2$\\

\hline

$3\mathfrak{su}(2)$ & $SU(2)^3$ & $=Q^{1,1,1}\not\cong S^7/\Gamma$
& \\

\hline
\end{tabular}
\end{center}

In the above table, $\Gamma$ denotes an arbitrary discrete subgroup of $O(6)$ or $O(8)$.
\end{Le}

The following fact will help us to prove the lemma:

\begin{Le} \label{U1Kuerzen}
Let $G$ be a Lie group and $H$ be a closed subgroup of $G$. Moreover, let $T\subseteq G\times U(1)$ be isomorphic to $U(1)$
such that $T\not\subseteq G$, $T\cap H = \{e\}$, and $T$ and $H$ commute. In this situation, $G/H$ is a $|\Gamma|$-fold
covering of $(G\times U(1))/(H\times T)$, where $\Gamma=T\cap G$.
\end{Le}

\begin{proof} A possible covering map is given by

\begin{equation}
\begin{split}
\pi & : G/H \rightarrow (G\times U(1))/(H\times T) \\
\pi(gH) & := (g,e)H\times T \\
\end{split}
\end{equation}
\end{proof}

We are now able to prove Lemma \ref{SingOrbQ111}:

\begin{proof}
The Lie algebra $\mathfrak{k}$ is a $2\mathfrak{u}(1)$-module. Since $\mathfrak{m}$ decomposes into pairwise inequivalent
$2\mathfrak{u}(1)$-submodules, the modules $V$ with $2\mathfrak{u}(1)\subseteq V\subseteq 3\mathfrak{su}(2)$ can be easily
classified. After we have checked for each $V$ if it is closed under the Lie bracket, we obtain the $\mathfrak{k}$ from the
table of the lemma. Our statements on the topology of $SU(2)^3/K$ can be proven with help of the Hopf fibration $SU(2)/U(1)
\cong S^2$. The remaining part of the proof is to check for each $K$ if $K/U(1)^2$ is a sphere. The first non-trivial case is
where $K\cong SU(2)\times U(1)^2$. In this situation $K\subseteq SU(2)^3$ is given by:

\begin{equation}
\{(A,Q_{\phi},Q_{\psi}) | A\in SU(2),\: \phi,\psi\in [0,2\pi)\}\:,
\end{equation}

where

\begin{equation*}
Q_{\phi}:=\left(\,\begin{array}{cc}
e^{i\phi} & 0 \\
0 & e^{-i\phi} \\
\end{array}\,\right)\:.
\end{equation*}

$U(1)^2$ can be described as:

\begin{equation}
\{(Q_{-\psi-\phi},Q_{\phi},Q_{\psi}) | \phi,\psi\in [0,2\pi)\}\:.
\end{equation}

We apply Lemma \ref{U1Kuerzen} to our situation. The intersection of the two-dimension\-al group $U(1)^2$ which we divide out
with the semisimple part of $K$ is trivial. Therefore, we have $K/U(1)^2 = (SU(2)\times U(1)^2)/U(1)^2\cong (SU(2)\times
U(1))/U(1)\cong SU(2)\cong S^3$. Next, we assume that $K = SU(2)^2\times U(1)$. For similar reasons as above, $K/U(1)^2$
can be described as $SU(2)^2/U(1)$ with

\begin{equation}
U(1):=\{(Q_{-\phi}, Q_{\phi}) | \phi\in [0,2\pi)\}\:.
\end{equation}

$S^3\times S^3$ is a circle bundle over $K/U(1)^2$. If $K/U(1)^2$
was covered by $S^5$, we would have:

\begin{equation}
\ldots\rightarrow \pi_3(S^1) \rightarrow \pi_3(S^3\times S^3) \rightarrow\pi_3(S^5/\Gamma)\rightarrow\ldots\:.
\end{equation}

Since the higher homotopy groups of $S^5$ and $S^5/\Gamma$ coincide, the above exact sequence becomes:

\begin{equation}
\ldots\rightarrow \{0\} \rightarrow \mathbb{Z}^2 \rightarrow\{0\}\rightarrow\ldots\,.
\end{equation}

which is impossible.  We finally consider the case $K=SU(2)^3$.  $Q^{1,1,1}$ is a circle bundle over $SU(2)^3/U(1)^3=S^2\times
S^2\times S^2$. Therefore, we obtain the following exact sequence:

\begin{equation}
\ldots \rightarrow \pi_2(S^7/\Gamma)\rightarrow \pi_2((S^2)^3)\rightarrow\pi_1(S^1) \rightarrow \pi_1(S^7/\Gamma) \rightarrow
\ldots\:,
\end{equation}

which can be explicitly written as

\begin{equation}
\ldots \rightarrow \{0\} \rightarrow \mathbb{Z}^3\rightarrow \mathbb{Z} \rightarrow \Gamma\rightarrow \ldots\:.
\end{equation}

Since $\Gamma$ is finite, this exact sequence is impossible, too, and $Q^{1,1,1}$ is not homeomorphic to $S^7/\Gamma$.
\end{proof}

Lemma (\ref{SingOrbQ111}) yields two sets of initial conditions at the singular orbit:

\begin{enumerate}
     \item If the singular orbit is $S^2\times S^2\times S^2$, we have $f(0)=0$ and $a(0),b(0),$ $c(0)\neq 0$.
     \item If the singular orbit is $S^2\times S^2$, we have $f(0)=0$ and exactly two of the three initial values $a(0)$,
     $b(0)$, and $c(0)$ are non-zero. Without loss of generality, we assume that $a(0)=0$.
\end{enumerate}

We check for each of the two cases if the metric can be smoothly
extended to the singular orbit and start with the first one.
According to the considerations which we have made at the end of
Section \ref{2ndSection}, the metric is smooth if the following
conditions are satisfied:

\begin{enumerate}
    \item $a$, $b$, $c$, and $f$ are analytic.
    \item The metric on $\mathfrak{p}=\text{span}(e_1,\ldots,e_6)$ is $U(1)^3$-invariant for all $t$.
    \item $a$, $b$, and $c$ are even and $f$ is odd.
    \item The length of the collapsing circle is $2\pi t + O(t^2)$ for small values of $t$.
\end{enumerate}

It follows from (\ref{SolutionfQ1111}) and (\ref{SolutionfQ1112}) that the first condition is satisfied. Since $\mathfrak{p}$
splits into $V_1\oplus V_2\oplus V_3$ with respect to $U(1)^3$, $a(t)$, $b(t)$, and $c(t)$ can be chosen arbitrarily and the
second condition is satisfied. Let $(a(t),b(t),c(t),f(t))$ be a solution of (\ref{HolRedQ111}). It is easy to see that
$(a(-t),b(-t),c(-t),-f(-t))$ is another solution of (\ref{HolRedQ111}), whose value at $t=0$ is the same. Since
(\ref{HolRedQ111}) has a unique solution for any initial metric on the singular orbit $S^2\times S^2\times S^2$, the third
condition is satisfied, too.

Let $\gamma:[0,\tfrac{4\pi}{3}]\rightarrow U(1)^3/U(1)^2_{1,1,1}$ be the loop with $\gamma(\theta):= S_{\theta}
U(1)^2_{1,1,1}$, where $S_{\theta}$ is defined by (\ref{DefSQ111}). As before, $S$ is the one-dimensional subgroup of
$SU(2)^3$ which consists of all $S_{\theta}$. Since the intersection $S\cap U(1)^2_{1,1,1}$ is isomorphic to $\mathbb{Z}_3$,
$\gamma$ winds around the collapsing circle $U(1)^3/U(1)^2_{1,1,1}$ exactly once. Moreover, we have $|\gamma'(0)| = |e_7| =
|f(t)|$ and the length of the collapsing circle thus is $\tfrac{4\pi}{3}|f(t)|$. The fourth smoothness condition is therefore
equivalent to $|f'(0)|=\tfrac{3}{2}$. If we insert $f(0)=0$ into (\ref{HolRedQ111}), we obtain $f'(0)=-3$ and see that the
metric has a singularity at the singular orbit. More precisely, the length of the collapsing circle is twice as long as it
should be. We can think of the singularity as "the opposite" of a conical singularity. We turn our attention to the second set
of initial conditions. As before, we have the following sufficient smoothness conditions:

\begin{enumerate}
    \item $a$, $b$, $c$, and $f$ are analytic.
    \item The metric on $\mathfrak{p}=\text{span}(e_3,e_4,e_5,e_6)$ is $U(1)^2\times SU(2)$-invariant for all $t$.
    \item $b$ and $c$ are even functions while $a$ and $f$ are odd functions.
    \item The length of any great circle on the collapsing sphere is $2\pi t + O(t^2)$ for small values of $t$.
\end{enumerate}

For the similar reasons as in the previous case, the first two conditions are satisfied for any choice of the initial values
$b(0)$ and $c(0)$. We can also prove the third smoothness condition by the same method as before, since
$(a(t),b(t),c(t),f(t))\mapsto (-a(-t),b(-t),c(-t),-f(-t))$ is another symmetry of the system (\ref{HolRedQ111}). In order to
check the fourth smoothness condition, we investigate the great circles on the collapsing sphere $S^3$. Since the metric on
$S^3$ is determined by the values of $a$ and $f$, it suffices to consider the great circles in the directions of $e_1$ and
$e_7$. By the same arguments as before, we obtain the old condition $|f'(0)|=\tfrac{3}{2}$ and $|a'(0)|=\tfrac{1}{2}$ as
smoothness conditions. We cannot directly insert $a(0)=f(0)=0$ into (\ref{HolRedQ111}), since those equations contain terms of
type $\tfrac{f}{a}$. However, we can apply l'H\^{o}pital's rule and obtain a system of two equations which yields
$a'(0)=\pm\tfrac{1}{2}$ and $f'(0)=-\tfrac{3}{2}$. All in all, we have proven that the metrics with singular orbit $S^2\times
S^2$ are smooth. We finally sum up the results of this section:

\begin{Th}
Let $(M,\Omega)$ be a $\text{Spin}(7)$-manifold with a
cohomogeneity one action of $SU(2)^3$ which preserves $\Omega$. In
this situation, the following statements are true:

\begin{enumerate}
   \item The principal orbits are $SU(2)^3$-equivariantly diffeomorphic to $Q^{1,1,1}$.
   \item The metric $g$ which is associated to $\Omega$ is Ricci-flat and K\"ahler and its holonomy is all of $SU(4)$.
   \item The restriction of $g$ to a principal orbit can be written as (\ref{Q111GenMetric}). The coefficient functions $a$,
   $b$, $c$, and $f$ satisfy the equations (\ref{HolRedQ111}), whose solutions are described by (\ref{SolutionfQ1111}) and
   (\ref{SolutionfQ1112}). The K\"ahler form is given by (\ref{EtaQ111}).
\end{enumerate}

If $M$ has a singular orbit, which has to be the case if $(M,g)$
is complete, it is $S^2\times S^2\times S^2$ or $S^2\times S^2$.
Any $SU(2)^3$-invariant metric on the singular orbit can be
extended to a unique complete cohomogeneity one metric with holonomy
$SU(4)$. For any choice of the singular orbit and its metric, $(M,g)$
is asymptotically conical. If the singular orbit is
$S^2\times S^2\times S^2$, the metric is not differentiable at the
singular orbit. If the singular orbit is $S^2\times S^2$, the
metric is smooth at the singular orbit.
\end{Th}

\begin{Rem}
The metrics with holonomy $SU(4)$ which we have constructed have
also been considered by Cveti\v{c}, Gibbons, L\"u, and Pope in
\cite{Cve03}. In that paper, the authors construct Ricci-flat
K\"ahler-metrics on holomorphic vector bundles over a product of
several Einstein-K\"ahler manifolds. As a special case, they
obtain the same metrics as the author. The authors also prove that
the metrics are non-compact and complete away from the singular
orbit. In \cite{Cve01}, the authors investigate the same metrics
in another context. Both papers are based on earlier works by
Berard-Bergery \cite{BerBer}, Page and Pope \cite{PagePope}, and
Stenzel \cite{Stenzel}. The examples with singular orbit
$S^2\times S^2\times S^2$ have also been considered by Herzog and
Klebanov \cite{Herzog}. Although the metrics which we have
constructed are known, the proof that any parallel
$\text{Spin}(7)$-manifold with a cohomogeneity one action of
$SU(2)^3$ is one of our examples is new. Our results thus prove
that it is impossible to deform the K\"ahler examples into metrics
with holonomy $\text{Spin}(7)$ without loosing the
$SU(2)^3$-symmetry. Moreover, it follows from Remark
\ref{SU4LemmaRemark} that any (not necessarily parallel)
$\text{Spin}(7)$-structure of cohomogeneity one with principal
orbit $Q^{1,1,1}$ has to reduce to an $SU(4)$-structure. This fact
has not been mentioned in the literature before, either. All in
all, we hope to have introduced an interesting, more algebraic
approach to the issue of special cohomogeneity one metrics with
principal orbit $Q^{1,1,1}$.
\end{Rem}

\section{Metrics with principal orbit $M^{k,l,0}$}
\label{M110}

In this section, we consider cohomogeneity one manifolds whose
principal orbit is of type $(SU(3)\times SU(2))/(SU(2)\times
U(1))$. The semisimple part $SU(2)$ of the isotropy group shall be
embedded into the first factor of $SU(3)\times SU(2)$. We will
write any element of $SU(3)\times SU(2)$ as a block matrix in
$GL(5,\mathbb{R})$ consisting of a $3\times 3$- and a $2\times
2$-matrix.

Our first step is to classify all homogeneous spaces of type
$(SU(3)\times SU(2))/$ $(SU(2)\times U(1))$. We use the notation
of Castellani et al. \cite{Cast} and Fabbri et al. \cite{Fab1}
and denote these spaces by $M^{k,l,0}$.  In order to explain the
meaning of the three indices, we first consider more general
spaces of type $M^{k,l,m}=(SU(3)\times SU(2)\times U(1)) / (SU(2)
\times U(1)' \times U(1)'')$. We define the following
one-dimensional subalgebra of $\mathfrak{su}(3) \oplus
\mathfrak{su}(2) \oplus \mathfrak{u}(1)$:

\begin{equation}
\mathfrak{u}(1)_{k,l,m}^{'''}:= \left\{
\left(\,\begin{array}{cccccc} \hhline{---~~~}
\multicolumn{1}{|c}{\tfrac{k}{2}ix} & 0 & \multicolumn{1}{c|}{0}
&&&\\ \multicolumn{1}{|c}{0} & \tfrac{k}{2}ix &
\multicolumn{1}{c|}{0} &&&\\ \multicolumn{1}{|c}{0} & 0 &
\multicolumn{1}{c|}{-kix}\\ \hhline{-----~} &&&
\multicolumn{1}{|c}{-\tfrac{l}{2}ix} & \multicolumn{1}{c|}{0} &\\
&&& \multicolumn{1}{|c}{0} & \multicolumn{1}{c|}{\tfrac{l}{2}ix}
&\\ \hhline{~~~---} &&&&& \multicolumn{1}{|c|}{mix}\\
\hhline{~~~~~-}
\end{array}\,\right)\right.\left|
\begin{array}{l} \\ \\ \\ \\ \\ \\ \end{array}
x\in\mathbb{R}\right\}\:.
\end{equation}

The diagonal matrices in $\mathfrak{su}(3) \oplus \mathfrak{su}(2)
\oplus \mathfrak{u}(1)$ are a Cartan subalgebra $\mathfrak{t}$ of
that algebra. With help of the above definitions, we define an embedding
$\imath_{k,l,m}:\mathfrak{su}(2)\oplus\mathfrak{u}(1)'\oplus\mathfrak{u}(1)''\rightarrow
\mathfrak{su}(3) \oplus \mathfrak{su}(2) \oplus \mathfrak{u}(1)$:
$\mathfrak{su}(2)$ is embedded into $\mathfrak{su}(3)$
such that it leaves the subspace $\mathbb{C}^2$ of $\mathbb{C}^3$
invariant. $\imath_{k,l,m}(\mathfrak{u}(1)'\oplus
\mathfrak{u}(1)'')$ shall be the orthogonal complement of
$\mathfrak{u}(1)_{k,l,m}^{'''} \oplus (\mathfrak{su}(2)\cap
\mathfrak{t})$ in $\mathfrak{t}$ with respect to the Killing form. The
quotient of $SU(3)\times SU(2)\times U(1)$ by the Lie subgroup
whose Lie algebra is
$\imath_{k,l,m}(\mathfrak{su}(2)\oplus\mathfrak{u}(1)'\oplus\mathfrak{u}(1)'')$
is denoted by $M^{k,l,m}$. In fact, any coset space of type $(SU(3)\times
SU(2)\times U(1)) / (SU(2) \times U(1)' \times U(1)'')$ where
$SU(2)$ is embedded into $SU(3)$ can be identified with an
$M^{k,l,m}$.

The space $M^{k,l,m}$ is covered by $M^{k,l,0}$ (cf. Castellani \cite{Cast} for details). We will therefore assume from now on
that $m=0$. Since in this case the abelian factor of $SU(3) \times SU(2) \times U(1)$ is a subgroup of $U(1)'\times U(1)''$,
$M^{k,l,0}$ is diffeomorphic to a quotient of type $(SU(3)\times SU(2))/ (SU(2) \times U(1))$. Let

\begin{equation}
P:=\left(\, I_3\:,
\left(\,\begin{array}{cc}
0 & i \\
-i & 0 \\
\end{array}\,\right)\,\right)\in SU(3)\times SU(2)\:,
\end{equation}

and let $\phi_P:SU(3)\times SU(2)\rightarrow SU(3)\times SU(2)$ be defined by $\phi_P(Q):=PQP^{-1}$. $\phi_P$ maps
$\imath_{k,l,0}(SU(2)\times U(1))$ into $\imath_{k,-l,0}(SU(2)\times U(1))$. The spaces $M^{k,l,0}$ and $M^{k,-l,0}$ thus are
$SU(3)\times SU(2)$-equivariantly diffeomorphic. Since $M^{k,l,0}$ and $M^{-k,-l,0}$ are the same, too, we can assume without
loss of generality that $(k,l)\in\mathbb{N}_0 \times \mathbb{N}_0 \setminus \{(0,0)\}$.

We prove which $M^{k,l,0}$ admit an $SU(3)\times SU(2)$-invariant $G_2$-structure. There exists a subalgebra of
$\mathfrak{g}_2$ which is isomorphic to $\mathfrak{su}(2)$ and acts irreducibly on $\mathbb{H}\epsilon$ and trivially on
$\text{Im}(\mathbb{H})$. More precisely, there exists an $\mathfrak{su}(2)$-equivariant map from $\mathbb{C}^2$ to
$\mathbb{H}\epsilon$ which maps the canonical basis of the real vector space $\mathbb{C}^2$ to
$(\epsilon,i\epsilon,k\epsilon,j\epsilon)$. Up to conjugation, there exists a unique one-dimensional subalgebra of
$\mathfrak{g}_2$ which commutes with $\mathfrak{su}(2)$. We thus have constructed a subalgebra which is isomorphic to
$\mathfrak{su}(2)\oplus \mathfrak{u}(1)$. By an analo\-gous calculation as in Section \ref{Q111}, we see that the action of this
algebra on $\text{Im}(\mathbb{O})$ is equivalent to the isotropy action if and only if $k=l=\pm 1$. All in all, we have shown
that the only principal orbit we have to consider is $M^{1,1,0}$.

We remark that the connected subgroup of $G_2$ which has $\mathfrak{su}(2)\oplus \mathfrak{u}(1)$ as Lie algebra is not
isomorphic to $SU(2)\times U(1)$, but to $U(2)$. The kernel of the isotropy representation of $SU(2)\times U(1)$ is
$\mathbb{Z}_2$. The group which acts effectively on the tangent space of $M^{1,1,0}$ therefore is $U(2)$ and there is no
contradiction to Lemma \ref{G2HomLemma}.

Next, we classify the $SU(3)\times SU(2)$-invariant metrics and $G_2$-structures on $M^{1,1,0}$. As in the previous section, we
identify the tangent space of $M^{1,1,0}$ with the complement $\mathfrak{m}$ of $\mathfrak{su}(2)\oplus\mathfrak{u}(1)$ in
$\mathfrak{su}(3)\oplus\mathfrak{su}(2)$ with respect to the metric $q(X,Y):=-\text{tr}(XY)$. We fix the following basis of
$\mathfrak{su}(3) \oplus \mathfrak{su}(2)$:

\begin{tabular}{ll}
$e_1:=(E_{13}-E_{31},0)$ & $e_2:=(iE_{13}+E_{31},0)$ \\
$e_3:=(E_{23}-E_{32},0)$ & $e_4=(iE_{23}+iE_{32},0)$ \\
$e_5:=(0,E_{12}-E_{21})$ & $e_6:=(0,iE_{12}+iE_{21})$ \\
$e_7:=(\tfrac{1}{2}iE_{11} + \tfrac{1}{2}iE_{22} - iE_{33},-\tfrac{1}{2}iE_{11}+\tfrac{1}{2}iE_{22})$ &
$e_8:=(E_{12}-E_{21},0)$ \\
$e_9:=(iE_{12}+iE_{21},0)$ & $e_{10}:=(iE_{11}-iE_{22},0)$ \\
$e_{11}:=(\tfrac{1}{3}iE_{11} + \tfrac{1}{3}iE_{22} - \tfrac{2}{3}iE_{33},iE_{11} - iE_{22})$ & \\
\end{tabular}

In the above table $E_{ij}$ denotes the $3\times 3$- ($2\times 2$-)matrix with a "$1$" in the $i$th row and $j$th column. All
other coefficients of $E_{ij}$ shall be zero. $(e_1,\ldots,e_7)$ is a basis of $\mathfrak{m}$ and $(e_8,\ldots,e_{11})$ is a
basis of the isotropy algebra $\mathfrak{su}(2) \oplus \mathfrak{u}(1)$. $\mathfrak{m}$ splits with respect to
$\mathfrak{su}(2) \oplus \mathfrak{u}(1)$ into the following irreducible submodules:

\begin{eqnarray}
V_1 & := & \text{span}(e_1,e_2,e_3,e_4) \\
V_2 & := & \text{span}(e_5,e_6) \\
V_3 & := & \text{span}(e_7)
\end{eqnarray}

As in Section \ref{Q111}, any $SU(3)\times SU(2)$-invariant metric on $M^{1,1,0}$ can be identified with a $q$-symmetric,
positive definite, $\mathfrak{su}(2) \oplus \mathfrak{u}(1)$-equivariant endomorphism of $\mathfrak{m}$. With help of Schur's
Lemma, we see that the invariant metrics $g$ are precisely those with:

\begin{equation}
\label{M110GenMetric}
\begin{split}
g = & a^2 (e^1\otimes e^1 + e^2\otimes e^2 + e^3\otimes e^3 + e^4\otimes e^4) + b^2(e^5\otimes e^5 + e^6\otimes e^6)\\
\quad & + c^2(e^7\otimes e^7) \\
\end{split}
\end{equation}

with $a,b,c\in\mathbb{R}\setminus\{0\}$. For the same reasons as in the previous section, we allow $a$, $b$, and $c$ to take
negative values, too.

We fix an orientation and an $SU(3)\times SU(2)$-invariant metric
$g$ on $M^{1,1,0}$. Let $\omega$ be an $SU(3)\times
SU(2)$-invariant $G_2$-structure such that its associated metric
is $g$ and $\omega\wedge\ast\omega$ is a positive volume form. Our
aim is to describe the set $M$ of all such $G_2$-structures. Any
invariant $G_2$-structure can be described by an
$\mathfrak{su}(2)\oplus\mathfrak{u}(1)$-equivariant map
$\psi:\mathfrak{m} \rightarrow\text{Im}(\mathbb{O})$. Let
$2\mathfrak{u}(1)$ be a Cartan subalgebra of $\mathfrak{su}(2)
\oplus \mathfrak{u}(1)$. $\psi$ is also a
$2\mathfrak{u}(1)$-equivariant map. The action of
$2\mathfrak{u}(1)$ on $\mathfrak{m}$ is equivalent to the action of
the Cartan subalgebra of $\mathfrak{g}_2$ on $\text{Im}(\mathbb{O})$
and thus to the isotropy action in Section \ref{Q111}. $M$ can therefore
be considered as a subset of the $U(1)$-orbit from that section. $U(1)$
acts on $\mathfrak{m}$ such that its representation is given by

\begin{equation}
\label{ActionNormM110} T:=\left\{\left(\,
\begin{array}{cccc}
\hhline{-~~~} \multicolumn{1}{|c|}{R_{\theta}} &&& \\
\hhline{--~~} & \multicolumn{1}{|c|}{R_{\theta}} && \\
\hhline{~--~} && \multicolumn{1}{|c|}{R_{\theta}} & \\
\hhline{~~--} &&& \multicolumn{1}{|c|}{1} \\ \hhline{~~~-}
\end{array}\,\right)\right.
\left|\begin{array}{l}  \\ \\ \\ \\ \end{array}
\theta\in\mathbb{R}\right\}\:,
\end{equation}

with respect to the basis $(e_1,\ldots,e_7)$. $R_{\theta}$ again
denotes the rotation in the plane around an angle of $\theta$.
Conjugation by any element of $T$ leaves not only
$2\mathfrak{u}(1)$ but also $\mathfrak{su}(2) \oplus
\mathfrak{u}(1)$ invariant. Moreover, $T$ does not preserve
$\omega$. Otherwise, the Lie algebra of $T$ and $\mathfrak{su}(2)
\oplus \mathfrak{u}(1)$ would generate a subalgebra of rank $3$ of
$\mathfrak{g}_2$. Therefore, the action of $T$ generates the set
of all invariant $G_2$-structures which have the same extension to
an $SO(7)$-structure as $\omega$. These $G_2$-structures can also
be obtained as the pull-back of $\omega$ by certain isometries of
$M^{1,1,0}$. We consider the group:

\begin{equation}
S:=\left\{\left(\,
\begin{array}{lllll}
\hhline{---~~} \multicolumn{1}{|c}{e^{\tfrac{i\theta}{2}}} & 0 &
\multicolumn{1}{c|}{0} &&\\ \multicolumn{1}{|c}{0} &
e^{\tfrac{i\theta}{2}} & \multicolumn{1}{c|}{0} &&\\
\multicolumn{1}{|c}{0} & 0 & \multicolumn{1}{c|}{e^{-i\theta}}
&&\\ \hhline{-----} &&&
\multicolumn{1}{|c}{e^{-\tfrac{i\theta}{2}}} &
\multicolumn{1}{c|}{0}\\ &&& \multicolumn{1}{|c}{0} &
\multicolumn{1}{c|}{e^{\tfrac{i\theta}{2}}}\\ \hhline{~~~--}
\end{array}
\,\right)=:S_{\theta}\right.\left|
\begin{array}{l} \\ \\ \\ \\ \\ \\ \\ \end{array}
\theta\in\mathbb{R}\right\}\:.
\end{equation}

$S$ is a subgroup of $\text{Norm}_{SU(3)\times SU(2)}(SU(2)\times
U(1))$. Conjugation by $S_{\theta}$ therefore induces a
well-defined diffeomorphism $f_{\theta}$ of $M^{1,1,0}$. We prove
that $f_{\theta}$ is even an isometry. $(df_{\theta})_e$ is
determined by the adjoint action of $S_{\theta}$ restricted to
$\mathfrak{m}$. Since $S$ is generated by $e_7$, we only have to
prove that $\text{ad}_{e_7}|_{\mathfrak{m}}$ is skew-symmetric
with respect to $g$. We obtain by a straightforward calculation:

\begin{equation}
\text{ad}_{e_7}|_{\mathfrak{m}}=\left(\,
\begin{array}{lllllll}
\hhline{--~~~~~} \multicolumn{1}{|c}{0} &
\multicolumn{1}{c|}{-\tfrac{3}{2}} &&&&&\\
\multicolumn{1}{|c}{\tfrac{3}{2}} & \multicolumn{1}{c|}{0} &&&&&\\
\hhline{----~~~} && \multicolumn{1}{|c}{0} &
\multicolumn{1}{c|}{-\tfrac{3}{2}} &&&\\ &&
\multicolumn{1}{|c}{\tfrac{3}{2}} & \multicolumn{1}{c|}{0} &&&\\
\hhline{~~----~} &&&& \multicolumn{1}{|c}{0} &
\multicolumn{1}{c|}{1} &\\ &&&& \multicolumn{1}{|c}{-1} &
\multicolumn{1}{c|}{0} &\\ \hhline{~~~~---} &&&&&&
\multicolumn{1}{|c|}{0}\\ \hhline{~~~~~~-}
\end{array}
\,\right)
\end{equation}

with respect to $(e_1,\ldots,e_7)$. This endomorphism is
skew-symmetric with respect to any metric of type
(\ref{M110GenMetric}). Moreover, it is orientation-preserving.
Unfortunately, the above matrix does not coincide with
(\ref{ActionNormM110}). Nevertheless, it generates the same family
of $G_2$-structures as $T$. The reason for this is that $T$ and
the action of $S$ on $\mathfrak{m}$ are both contained in
$\text{Norm}_{SO(7)}(\mathfrak{su}(2) \oplus \mathfrak{u}(1))$,
but not in $\text{Norm}_{G_2}(\mathfrak{su}(2) \oplus
\mathfrak{u}(1))$. In fact, we could have chosen $S$ as any
connected one-dimensional subgroup of the standard maximal torus
of $SU(3)\times SU(2)$ which is not contained in $SU(2)\times
U(1)$.

Again, let $g$ be an arbitrary $SU(3)\times SU(2)$-invariant metric on $M^{1,1,0}$. Furthermore, we assume without loss of
generality that $\tfrac{\partial}{\partial t}$ is of unit length. Our next aim is to construct a basis $(f_i)_{0\leq i\leq 7}$
of the tangent space which yields an $SU(3)\times SU(2)$-invariant $\text{Spin}(7)$-structure such that the restriction of its
associated metric to the principal orbit is $g$. $(f_i)_{1\leq i\leq 7}$ has to be orthonormal with respect to $g$. Moreover,
it should be possible to identify the $\mathfrak{su}(2) \oplus \mathfrak{u}(1)$-modules $\mathfrak{m}$ and
$\text{Im}(\mathbb{O})$ with each other, such that $(f_i)_{1\leq i\leq 7}$ is mapped to the basis $(i,\ldots,k\epsilon)$.
Motivated by these considerations, we choose:

\begin{equation}
\label{CBasisM110}
\begin{array}{llll}
f_0:=\tfrac{\partial}{\partial t} & f_1:=\tfrac{1}{c}e_7 &
f_2:=\tfrac{1}{b} e_6 & f_3:=\tfrac{1}{b}e_5\\ &&&\\
f_4:=\tfrac{1}{a}e_1 & f_5:=\tfrac{1}{a}e_2 & f_6:=\tfrac{1}{a}e_4
& f_7:=\tfrac{1}{a}e_3\\
\end{array}
\end{equation}

For similar reasons as in the previous section, this is the most
general basis which we have to consider. The four-form $\Omega$
which is determined by (\ref{CBasisM110}) is given by:

\begin{equation}
\begin{split}
\Omega & =-a^4\:e^{1234}+a^2b^2\:e^{1256}+a^2b^2\:e^{3456}\\
&\quad
+a^2bc\:e^{1367}-a^2bc\:e^{1457}-a^2bc\:e^{2357}-a^2bc\:e^{2467}\\
&\quad -a^2b\:e^{135}\wedge dt-a^2b\:e^{146}\wedge
dt-a^2b\:e^{236}\wedge dt +a^2b\:e^{245}\wedge dt\\ &\quad
-a^2c\:e^{127}\wedge dt-a^2c\:e^{347}\wedge dt+b^2c\:e^{567}\wedge
dt
\\
\end{split}
\end{equation}

After having calculated $d\Omega$, we  are able to express the
condition $d\Omega=0$ as a system of ordinary differential
equations for the functions $a$, $b$, and $c$:

\begin{equation}
\label{HolRedM110}
\begin{split}
\frac{a'}{a} & = \frac{3}{8}\frac{c}{a^2}\\ \frac{b'}{b} & =
\frac{1}{4}\frac{c}{b^2}\\ \frac{c'}{c} & =
8\frac{1}{c}-\frac{1}{4}\frac{c}{b^2} -\frac{3}{4}\frac{c}{a^2}\\
\end{split}
\end{equation}

The above system can be solved explicitly. We denote the initial
values $a(0)$, $b(0)$, and $c(0)$ by $a_0$, $b_0$, and $c_0$.
Furthermore, we define $C(t):=\int_0^t c(s)ds$. The first two
equations of (\ref{HolRedM110}) can be simplified to:

\begin{equation}
\label{SolutionabM110}
\begin{split}
a^2 & = \tfrac{3}{4}C + a_0^2\\ b^2 & = \tfrac{1}{2}C + b_0^2\\
\end{split}
\end{equation}

As in Section \ref{Q111}, we insert the two above equations into
the last one of (\ref{HolRedM110}). After that, we are able to
deduce an equation for the function
$\widetilde{c}(t):=c(C^{-1}(t))$ and find the following solution
of our initial value problem:

\begin{equation}
\label{SolutioncM110}
\begin{split}
\widetilde{c}(t)^2 & = \frac{32a_0^4b_0^2c_0^2}{9(t + 2b_0^2)(t +
\tfrac{4}{3}a_0^2)^2}\\ &\quad +
\frac{8}{(t + 2b_0^2)(t + \tfrac{4}{3}a_0^2)^2}
\int_0^t(s+2b_0^2)(s+\tfrac{4}{3}a_0^2)^2ds\:.\\
\end{split}
\end{equation}

$\widetilde{c}(t)$ is $c\sqrt{t} + O(1)$ for $t\rightarrow\infty$
and a $c\in\mathbb{R} \setminus \{0\}$. We can apply the same
arguments as in Section \ref{Q111} and see that the metric is
asymptotically conical. The cone which the metric approaches is
given by

\begin{equation}
\label{M110NP}
a(t)^2=\tfrac{3}{4}t^2\:,\quad b(t)^2=\tfrac{1}{2}t^2\:, \quad\text{and}\quad c(t)=2t\:.
\end{equation}

The base of the cone carries a nearly parallel $G_2$-structure.
Conversely, any nearly parallel $G_2$-structure on $M^{1,1,0}$ can
be described by $|a|=\tfrac{\sqrt{3}}{2}|t|$,
$|b|=\tfrac{1}{\sqrt{2}}|t|$, and $c=2t$ for a $t\in\mathbb{R}
\setminus \{0\}$.

Since there is an isometric $U(1)$-action which preserves the
$\text{Spin}(7)$-structure, the holonomy is a subgroup of $SU(4)$.
The existence of a parallel vector field can be excluded by the
same arguments as in Section \ref{Q111}. The holonomy therefore is
either $SU(4)$ or contained in $Sp(2)$. Any $SU(3)\times
SU(2)$-invariant two-form on $M^{1,1,0}$ is a linear combination
of:

\begin{equation}
e^{12} + e^{34}\:,\quad e^{56}\:, \quad e^7\wedge dt\:.
\end{equation}

Let $\eta$ be a K\"ahler form on $M$. Analogously to the previous
section, $\eta$ has to satisfy:

\begin{equation}
\eta = \epsilon_1 a^2\: e^{12} + \epsilon_1 a^2\: e^{34} +
\epsilon_2 b^2\: e^{56} + \epsilon_3 c\: e^7 \wedge
dt\quad\text{with}\quad\epsilon_1,\:\epsilon_2,\: \epsilon_3
\in\{-1,1\}\:.
\end{equation}

The condition $d\eta=0$ yields:

\begin{equation}
\frac{a'}{a}=\epsilon_1\epsilon_3\frac{3}{8}\frac{c}{a^2}\quad\text{and}\quad
\frac{b'}{b}=-\epsilon_2\epsilon_3\frac{1}{4}\frac{c}{b^2}\:.\\
\end{equation}

If we choose $\epsilon_1=\epsilon_3=1$ and $\epsilon_2=-1$, the
above equations are a consequence of (\ref{HolRedM110}) and the
K\"ahler form becomes:

\begin{equation}
\label{EtaM110} \eta = a^2\: e^{12} + a^2\: e^{34} - b^2\: e^{56} + c\: e^7 \wedge dt\:.
\end{equation}

For any other choice of $\epsilon_1$, $\epsilon_2$, and
$\epsilon_3$, $\eta$ would change its sign or we would obtain a
contradiction to (\ref{HolRedM110}). We thus have proven that the
holonomy is all of $SU(4)$. As in the previous section, our next
step is to classify the possible singular orbits:

\begin{Le}
Let $SU(2)\times U(1)$ be embedded into $SU(3)\times SU(2)$ such
that the quotient $(SU(3)\times SU(2))/(SU(2)\times U(1))$ is
$M^{1,1,0}$. Furthermore, let $K$ with $SU(2)\times U(1)
\subsetneq K\subseteq SU(3)\times SU(2)$ be a closed, connected
subgroup. We denote the Lie algebra of $K$ by $\mathfrak{k}$. In
this situation, $\mathfrak{k}$ and $K$ can be found in the table
below. Furthermore, $K/(SU(2)\times U(1))$ and $(SU(3)\times
SU(2))/K$ satisfy the following topological conditions:

\begin{center}
\begin{tabular}{|l|l|l|l|}
\hline

$\mathfrak{k}$ & $K$ & $K/(SU(2)\times U(1))$ & $(SU(3)\times
SU(2))/K$ \\

\hline \hline

$\mathfrak{su}(2)\oplus2\mathfrak{u}(1)$ & $U(2)\times U(1)$ &
$\cong S^1$ & $\cong \mathbb{CP}^2\times S^2$\\

\hline

$2\mathfrak{su}(2)\oplus\mathfrak{u}(1)$ & $U(2)\times SU(2)$ &
$\cong S^3$ & $\cong\mathbb{CP}^2$\\

\hline

$\mathfrak{su}(3)\oplus\mathfrak{u}(1)$ & $SU(3)\times U(1)$ &
$\cong S^5/\mathbb{Z}_3$ & $\cong S^2$\\

\hline

$\mathfrak{su}(3)\oplus\mathfrak{su}(2)$ & $SU(3)\times SU(2)$ &
$=M^{1,1,0}\not\cong S^7/\Gamma$ & \\

\hline
\end{tabular}
\end{center}

In the above table, $\Gamma\subseteq O(8)$ denotes an arbitrary discrete subgroup and the group $\mathbb{Z}_3$ which we divide
out is $\{\lambda\:\textnormal{Id}_{\mathbb{C}^3}|\lambda^3=1\}$.
\end{Le}

\begin{proof} The Lie algebras $\mathfrak{k}$ can be classified by
the same methods as in the proof of Lemma \ref{SingOrbQ111}.
$\mathbb{CP}^2$ can be described as $SU(3)/S(U(2)\times U(1))$.
With help of this fact, we are able to determine the topology of
$(SU(3)\times SU(2))/K$ in all three cases. In order to finish the
proof, we have to check if $K/(SU(2)\times U(1))$ is covered by a
sphere. If $K=U(2)\times U(1)$, this is true, since the circle is
the only one-dimensional, compact, connected, homogeneous space.

Next, we assume that $K=S(U(2)\times U(1))\times SU(2)$.
$\mathfrak{k}$ contains the semisimple part of the isotropy
algebra $\mathfrak{su}(2) \oplus \mathfrak{u}(1)$ as an ideal. We
can therefore cancel one factor of type $SU(2)$ and see that
$K/(SU(2)\times U(1))$ is covered by a space of type $(SU(2)
\times U(1))/U(1)$. According to Lemma \ref{U1Kuerzen}, the
universal cover of $K/(SU(2)\times U(1))$ is the sphere $S^3$. The
long exact homotopy sequence

\begin{equation}
\begin{split}
\ldots & \rightarrow \pi_2(K/(SU(2)\times U(1))) \rightarrow \pi_1(SU(2)\times U(1)) \rightarrow \pi_1(K) \\
& \rightarrow \pi_1(K/(SU(2)\times U(1))) \rightarrow \{0\}
\end{split}
\end{equation}

is in our situation given by

\begin{equation}
\ldots\rightarrow \{0\} \rightarrow \mathbb{Z} \overset{\pi_1(i)}{\longrightarrow} \mathbb{Z} \rightarrow \pi_1(K/(SU(2)\times
U(1))) \rightarrow \{0\}\:.
\end{equation}

The morphism $\pi_1(i)$ is induced by the inclusion $i$ of
$SU(2)\times U(1)$ into $K$. The homotopy class of the circle $\{
\text{diag}(e^{it},e^{it},e^{-2it},e^{3it},e^{-3it})|t\in[0,2\pi]\}$
generates $\pi_1(SU(2)\times U(1))$. This class is mapped by
$\pi_1(i)$ to the homotopy class of
$\{\text{diag}(e^{it},e^{it},e^{-2it},0,0)|t\in[0,2\pi]\}$, since
$SU(2)$ is simply connected.
$\{\text{diag}(e^{it},e^{it},e^{-2it})|t\in[0,2\pi]\}$ generates
the fundamental group of $S(U(2)\times U(1))$. $\pi_1(i)$
therefore is an isomorphism and $K/(SU(2)\times U(1))$ is a
sphere. We proceed to the case $K=SU(3)\times U(1)$: The
intersection of the abelian factor of $SU(2)\times U(1)$ with
$SU(3)\subseteq K$ is:

\begin{equation}
\left\{ \left(\,\begin{array}{lll} e^{\frac{2\pi ik}{3}} & 0 & 0
\\ 0 & e^{\frac{2\pi ik}{3}} & 0 \\ 0 & 0 & e^{-\frac{4\pi
ik}{3}} \\ \end{array}\,\right) =e^{\frac{2\pi
ik}{3}}\:\text{Id}_{\mathbb{C}^3} \right.\left|
\begin{array}{l} \\ \\ \\ \\ \end{array} k\in\mathbb{Z}\right\}
\:.
\end{equation}

We conclude with help of Lemma \ref{U1Kuerzen} that

\begin{equation}
K/(SU(2)\times U(1)) = (SU(3)/SU(2))/
\mathbb{Z}_3=S^5/\mathbb{Z}_3\:,
\end{equation}

where $\mathbb{Z}_3$ is generated by $e^{\frac{2\pi
i}{3}}\:\text{Id}_{\mathbb{C}^3}$. We finally assume that
$M^{1,1,0}$ is diffeomorphic to $S^7/\Gamma$. Since $M^{1,1,0}$ is
a circle bundle over $\mathbb{CP}^2\times S^2$, we have the
following exact sequence:

\begin{equation}
\ldots\rightarrow\underbrace{\pi_2(S^7/\Gamma)}_{=\{0\}}\rightarrow\underbrace{\pi_2(\mathbb{CP}^2\times S^2)}_{=
\mathbb{Z}^2}\rightarrow \underbrace{\pi_1(S^1)}_{=\mathbb{Z}}\rightarrow\ldots\:.
\end{equation}

There exists no injective group homomorphism from $\mathbb{Z}^2$
into $\mathbb{Z}$ and thus we have proven the lemma.
\end{proof}

Depending on the singular orbit, we have three types of initial conditions:

\begin{enumerate}
    \item If the singular orbit is $\mathbb{CP}^2\times S^2$, we have $c(0)=0$.
    \item If the singular orbit is $\mathbb{CP}^2$, we have $b(0)=c(0)=0$.
    \item If the singular orbit is $S^2$, we have $a(0)=c(0)=0$.
\end{enumerate}

The other initial values are non-zero. For all three cases, we have similar sufficient smoothness conditions:

\begin{enumerate}
    \item $a$, $b$, and $c$ are analytic.
    \item The metric on $\mathfrak{p}$ is $K$-invariant for all $t$.
    \item The functions which vanish at the singular orbit are odd and the other ones are even.
    \item If the singular orbit is $\mathbb{CP}^2\times S^2$, the length of the collapsing circle is $2\pi t + O(t^2)$ for
    small values of $t$. In the other cases, the sectional curvature of the collapsing sphere has to be $\tfrac{1}{t^2} +
    O(\tfrac{1}{t})$ for $t\rightarrow 0$.
\end{enumerate}

It follows from (\ref{SolutionabM110}) and (\ref{SolutioncM110})
that the first condition is satisfied. We consider $M^{1,1,0}$ as
a circle bundle over $\mathbb{CP}^2 \times S^2$. Any choice of
$a(t)\in\mathbb{R}\setminus\{0\}$ yields a multiple of the
Fubini-Study metric on $\mathbb{CP}^2$. Analogously, any
$b(t)\in\mathbb{R}\setminus\{0\}$ yields a metric with constant
sectional curvature on $S^2$. Therefore, the second condition is
satisfied, too. Since

\begin{eqnarray}
(a(t),b(t),c(t)) & \mapsto & (a(-t),b(-t),-c(-t)) \\
(a(t),b(t),c(t)) & \mapsto & (a(-t),-b(-t),-c(-t)) \\
(a(t),b(t),c(t)) & \mapsto & (-a(-t),b(-t),-c(-t))
\end{eqnarray}

are symmetries of the system (\ref{HolRedM110}), the third condition is also satisfied.

We finally check the fourth condition for each of our cases
separately and start with $\mathbb{CP}^2\times S^2$ as singular
orbit. Again, let $S$ be the subgroup of $SU(3)\times SU(2)$,
which is generated by $e_7$. The intersection of $S$ and the
isotropy group $SU(2)\times U(1)$ is isomorphic to $\mathbb{Z}_8$.
The smallest $t>0$ such that $\exp{(e_7 t)}$ is the unit element
is $4\pi$. By similar arguments as in Section \ref{Q111}, we have
$|c'(0)|=\tfrac{2\pi}{4\pi}\cdot 8=4$. Unfortunately, it follows
from (\ref{HolRedM110}) that $c'(0)=8$. The metric thus is not
smooth at the singular orbit and we have a singularity which is
similar to the singularity of the metrics with principal orbit
$Q^{1,1,1}$ and singular orbit $S^2\times S^2\times S^2$.

Next, we assume that the singular orbit is $\mathbb{CP}^2$. The collapsing sphere $S^3$ can be identified with the Lie group
$SU(2)$. The metric on $S^3$ with constant sectional curvature $1$ is given by $h(X,Y)=-\tfrac{1}{2}\text{tr}(XY)$ for all
$X,Y\in\mathfrak{su}(2)$. The following matrices are an orthonormal basis of $\mathfrak{su}(2)$ with respect to $h$:

\begin{equation}
\left(\,\begin{array}{cc} i & 0 \\ 0 & -i \\
\end{array}\,\right)\:,\quad\left(\,\begin{array}{cc} 0 & 1 \\ -1 & 0 \\
\end{array}\,\right)\:,\quad\left(\,\begin{array}{cc} 0 & i \\ i & 0 \\
\end{array}\,\right)\:.
\end{equation}

The second and the third matrix correspond to $e_5$ and $e_6$,
since it is the second factor of $SU(3)\times SU(2)$ which acts
irreducibly on $\mathfrak{p}^\perp$. The metric on the collapsing
sphere has sectional curvature $\tfrac{1}{t^2} + O(\tfrac{1}{t})$
only if $\left.\tfrac{\partial}{\partial t}\right|_{t=0}
\sqrt{g_t(e_5,e_5)} = \left.\tfrac{\partial}{\partial
t}\right|_{t=0} \sqrt{g_t(e_6,e_6)} = 1$ or equivalently if
$|b'(0)|=1$. Since the collapsing circle from the last case is
also a great circle of $S^3$, we again have $|c'(0)|=4$ as
smoothness condition. Any solution of (\ref{HolRedM110}) with
$b(0)=c(0)=0$ indeed satisfies $b'(0)\in\{-1,1\}$ and $c'(0)=4$.
The reason for the different values of $c'(0)$ in both cases is
that we now have to apply l'H\^{o}pital's rule. We remark that it
suffices to restrict ourselves to the case $b'(0)=1$, since we can
replace $b$ by $-b$ without changing (\ref{HolRedM110}). All in
all, we have proven that our metric is smooth if the singular
orbit is $\mathbb{CP}^2$.

We finally assume that the singular orbit is $S^2$. By similar
arguments as in the previous case, we obtain $|a'(0)|=1$ and
$|c'(0)|=4$ as smoothness conditions. The collapsing sphere is in
fact a space form of type $S^5/\mathbb{Z}_3$. Nevertheless, we
would obtain a smooth orbifold metric if our conditions were
satisfied. The system (\ref{HolRedM110}) yields $a'(0)=\pm 1$ and
$c'(0)=\tfrac{8}{3}$ if $a(0)=c(0)=0$. Our metric therefore has a
singularity. At the end of this section, we sum up our results:

\begin{Th}
Let $(M,\Omega)$ be a $\text{Spin}(7)$-orbifold with a
cohomogeneity one action of $SU(3)\times SU(2)$ which preserves
$\Omega$. We assume that all principal orbits are $SU(3)\times
SU(2)$-equivariantly diffeomorphic to a space of type $M^{k,l,0}$.
In this situation, the following statements are true:

\begin{enumerate}
   \item The principal orbits are $SU(3)\times SU(2)$-equivariantly diffeomorphic to $M^{1,1,0}$.
   \item The metric $g$ which is associated to $\Omega$ is Ricci-flat and K\"ahler and its holonomy is all of $SU(4)$.
   \item The restriction of $g$ to a principal orbit can be written as (\ref{M110GenMetric}). The coefficient functions $a$,
   $b$, and $c$ satisfy the equations (\ref{HolRedM110}), whose solutions are described by (\ref{SolutionabM110}) and
   (\ref{SolutioncM110}). The K\"ahler form is given by (\ref{EtaM110}).
\end{enumerate}

If $M$ has a singular orbit, which has to be the case if $(M,g)$
is complete, it is $\mathbb{CP}^2\times S^2$, $\mathbb{CP}^2$, or
$S^2$. Any $SU(3)\times SU(2)$-invariant metric on the singular
orbit can be extended to a unique complete cohomogeneity one
metric with holonomy $SU(4)$. For any choice of the singular orbit
and its metric, $(M,g)$ is asymptotically conical.

\begin{enumerate}
    \item If the singular orbit is $\mathbb{CP}^2\times S^2$, $M$
    is a manifold, but the metric $g$ is not smooth at the
    singular orbit.
    \item If the singular orbit is $\mathbb{CP}^2$, $M$ is a
    manifold and the metric is smooth.
    \item If the singular orbit is $S^2$, $M$ is an orbifold but
    not a manifold. At the singular orbit, the metric is not
    a smooth orbifold metric. Topologically, $M$ is a
    $\mathbb{C}^3/\mathbb{Z}_3$-bundle over $S^2$.
\end{enumerate}
\end{Th}

\begin{Rem}
The K\"ahler metrics with singular orbit $\mathbb{CP}^2\times S^2$
have been considered by Herzog and Klebanov in \cite{Herzog} and
by Cveti\v{c} et al. in \cite{Cve01}. The examples with singular
orbit $\mathbb{CP}^2$ or $S^2$ are mentioned by the same authors
in \cite{Cve01} and \cite{Cve03}. Our considerations prove that
any complete cohomogeneity one orbifold metric whose principal
orbits are $M^{1,1,0}$ and whose holonomy is a subgroup of
$\text{Spin}(7)$ is one of our examples. This result is, as far as
the author knows, new.
\end{Rem}


\begin{thebibliography}{999}
    \bibitem{Ach} Acharya, Bobby S.; Gukov, Sergei: M theory and Singularities of Exceptional Holonomy Manifolds. Phys. Rept.
    392, 121-189 (2004). Online available at: arXiv:hep-th/0409191.
    \bibitem{Alki} Alekseevskii, D.V.; Kimelfeld, B.N.: Structure of homogeneous Riemannian spaces with zero Ricci curvature.
    Functional Anal. i Prilozen 9(2),5-11 (1975). English translation: Functional Anal. Appl. 9,97-102 (1975).
    \bibitem{Baz} Bazaikin, Ya.V.: On the new examples of complete noncompact Spin(7)-holonomy metrics. Siberian Mathematical
    Journal Vol.48, No.1, 8-25 (2007).
    \bibitem{BerBer} Berard-Bergery, L.: Quelques exemples de varietes riemanniennes completes non compactes a courbure de
    Ricci positive. C.R. Acad. Sci. Paris Ser. A-B 302, 159-161 (1986).
    \bibitem{Bon} Bonan, E.: Sur les vari\'et\'es riemanniennes \`a groupe d'holonomie $G_2$ ou $\text{Spin}(7)$. C.R. Acad.
    Sci. Paris 262, 127-129 (1966).
    \bibitem{Br} Bryant, Robert: Metrics with exceptional holonomy. Ann. of Math. 126, 525-576 (1987).
    \bibitem{BrS} Bryant, Robert; Salamon, Simon: On the construction of some complete metrics with exceptional holonomy.
    Duke Mathematical Journal 58, 829-850 (1989).
    \bibitem{Cast} Castellani, L.; Romans, L.J.; Warner, N.P.: A classification of compactifying solutions for $d=11$
    supergravity. Nuclear Physics B241, 429-462 (1984).
    \bibitem{Cve01} Cveti\v{c}, M.; Gibbons, G.W.; L\"u, H.; Pope, C.N.: Supersymmetric non-singular fractional D2-branes and
    NS-NS 2-branes. Nucl. Phys. B 606 No.1-2, 18-44 (2001). Online available at: arXiv:hep-th/0101096.
    \bibitem{Cve02} Cveti\v{c}, M.; Gibbons, G.W.; L\"u, H; Pope, C.N.: Hyper-K\"ahler Calabi metrics, $L^2$ harmonic forms,
    resolved $M2$-branes, and $AdS_4/CFT_3$ correspondence. Nucl. Phys. B 617 No.1-3, 151-197 (2001). Online available at:
    arXiv:hep-th/0102185.
    \bibitem{Cve} Cveti\v{c}, M.; Gibbons, G.W.; L\"u, H.; Pope, C.N.: Cohomogeneity one manifolds of Spin(7) and $G_{2}$
    holonomy. Ann. Phys. 300 No.2, 139-184 (2002). Online available at: arXiv:hep-th/0108245.
    \bibitem{Cve03} Cveti\v{c}, M.; Gibbons, G.W.; L\"u, H.; Pope, C.N.: Ricci-flat metrics, harmonic forms and brane
    resolutions. Commun. Math. Phys. 232 No.3, 457-500 (2003). Online available at: arXiv:hep-th/0012011.
    \bibitem{DeKa} DeTurck, D.; Kazdan, J.: Some regularity theorems in Riemannian geometry. Ann. Scient. Ec. Norm. Sup. $4^o$
    s\'erie 14, 249-260 (1981).
    \bibitem{Esch} Eschenburg, J.-H.; Wang, McKenzie Y.: The Initial Value Problem for Cohomo\-geneity One Einstein Metrics. The
    Journal of Geometric Analysis Volume 10 Number 1, 109-137 (2000).
    \bibitem{Fab1} Fabbri, Davide; Fr\'e, Pietro; Gualtieri, Leonardo; Termonia, Piet: M-theory on
    $\text{AdS}_4\times M^{111}$: The complete $\text{Osp}(2|4)\times SU(3)\times SU(2)$ spectrum from harmonic analysis.
    Nucl. Phys. B 560. No. 1-3, 617-682 (1999). Online available at: arXiv:hep-th/9903036.
    \bibitem{Fer} Fern\'andez, M.; Gray, A.: Riemannian manifolds with structure
    group $G_2$. Annali di matematica pura ed applicata 132, 19-45 (1982).
    \bibitem{Fer2} Fern\'andez, M.: A classification of Riemannian manifolds with structure group Spin($7$). Annali di
    matematica pura ed applicata 143, 101-122 (1986).
    \bibitem{Guk} Gukov, Sergei; Sparks, James: M-theory on Spin(7) manifolds. Nucl. Phys. B 625 No.1-2, 3-69 (2002). Online
    available at: arXiv:hep-th/0109025.
    \bibitem{Herzog} Herzog, Christopher P.; Klebanov, Igor R.: Gravity duals of fractional branes in various dimensions.
    Phys. Rev. D 63, 126005 (2001). Online available at: arXiv:hep-th/0101020.
    \bibitem{Hitch} Hitchin, Nigel: Stable forms and special metrics. In: Fern\'andez, Marisa (editor) et al.: Global
    differential geometry: The mathematical legacy of Alfred Gray. Proceedings of the international congress on differential
    geometry held in memory of Professor Alfred Gray. Bilbao, Spain, September 18-23 2000. / AMS Contemporary Mathematical
    series 288, 70-89 (2001). Online available at: arXiv:math.DG/0107101.
    \bibitem{Joy} Joyce, Dominic D.: Compact Manifolds with Special Holonomy. New York 2000.
    \bibitem{Mostert} Mostert, Paul S.: On a compact Lie group acting on a manifold. Ann. Math. (2) 65, 447-455 (1957); errata
    ibid. 66, 589 (1957).
    \bibitem{PagePope} Page, D.N.; Pope, C.N.: Inhomogeneous Einstein metrics on complex line bundles. Class. Quantum Grav. 4
    No. 2., 213-225 (1987).
    \bibitem{Rei1} Reidegeld, Frank: Spaces admitting homogeneous $G_2$-structures. Preprint. Online available at:
    arXiv:0901.0652.
    \bibitem{Stenzel} Stenzel, Matthew B.: Ricci-flat metrics on the complexification of a compact rank one symmetric space.
    Manuscripta Mathematica 80 No. 2., 151-163 (1993).
    \bibitem{Wang} Wang, McKenzie Y.: Parallel Spinors and Parallel Forms. Ann. Global Anal. Geom. Vol. 7 No. 1, 59-68 (1989).
\end{thebibliography}
\end{document}